\documentclass[12pt]{amsart}
 
\usepackage[usenames,dvipsnames]{color}
\usepackage[english]{babel}
\usepackage{graphicx}
\usepackage{fullpage} 
\usepackage{tikz} 
\usepackage{hyperref}
\usepackage{amsmath,amsfonts,amssymb,amsthm}
\usepackage{environ}
\usepackage{mathtools}
\usepackage{tikz-cd}

\makeatletter
\let\@addpunct\@gobble
\newsavebox{\measure@tikzpicture}
\NewEnviron{scaletikzpicturetowidth}[1]{%
	\def\tikz@width{#1}%
	\begin{lrbox}{\measure@tikzpicture}%
		\BODY
	\end{lrbox}%
	\pgfmathparse{#1/\wd\measure@tikzpicture}%
	\BODY
}
\makeatother

\usetikzlibrary{er,positioning}
\usepackage{booktabs}
\usetikzlibrary{fit,shapes,backgrounds}
\usetikzlibrary{shapes.geometric}
\usetikzlibrary{shapes.arrows}
\usepackage{array}
\usepackage[all]{nowidow}

\newcommand{\ffids}{\ensuremath{\mathrm{(ffid^{[c,d]})}}}
\newcommand{\ffid}{\ensuremath{\mathrm{(ffid)}}}
\newcommand{\fid}{\ensuremath{\mathrm{(fid)}}}
\newcommand{\cfid}{\ensuremath{\mathrm{(cfid)}}}
\newcommand{\cid}{\ensuremath{\mathrm{(cid)}}}
\newcommand{\id}{\ensuremath{\mathrm{(id)}}}
\newcommand{\pfd}{\ensuremath{\mathrm{(pfd)}}}
\newcommand{\eph}{\ensuremath{\mathrm{(eph)}}}
\newcommand{\qtame}{\ensuremath{\mathrm{(qtame)}}}
\newcommand{\pers}{\ensuremath{\mathrm{(pm)}}}
\newcommand{\zm}{\ensuremath{\mathrm{(0)}}}
\newcommand{\kid}{\ensuremath{\operatorname{(\kappa-id})}}
\newcommand{\rid}{\ensuremath{\mathrm{(rid)}}}
\newcommand{\pd}{\ensuremath{\mathrm{(pd)}}}

\newcommand{\R}{\underline{\bf{R}}}
\newcommand{\vect}{\text{\underline{Vect}}_\gf}
\newcommand{\Zero}{0}

\numberwithin{equation}{section}

\theoremstyle{theorem}
	\newtheorem{theorem}[equation]{Theorem}
	\newtheorem{lemma}[equation]{Lemma}
	\newtheorem{corollary}[equation]{Corollary}
	\newtheorem{proposition}[equation]{Proposition}
	\newtheorem*{theorem*}{Theorem}

\theoremstyle{definition}
	\newtheorem{definition}[equation]{Definition}
	\newtheorem{example}[equation]{Example}

\theoremstyle{remark}
	\newtheorem{remark}[equation]{Remark}
	\newtheorem*{remark*}{Remark}

\newcommand{\Id}{1}

\newcommand{\To}{\Rightarrow}
\newcommand{\From}{\Leftarrow}
\newcommand{\xto}{\xrightarrow}
\newcommand{\incl}{\hookrightarrow}

\newcommand{\isom}{\cong}

\newcommand{\eps}{\varepsilon}
\newcommand{\gf}{\mathbf{k}}

\newcommand{\st}{\ \mid \ }
\newcommand{\A}{\mathcal{A}}
\newcommand{\B}{\mathcal{B}}
\newcommand{\C}{\mathcal{C}}
\newcommand{\D}{\mathcal{D}}

\newcommand\abs[1]{\lvert#1\rvert}
\newcommand\norm[1]{\lVert#1\rVert}

\DeclareMathOperator{\im}{im}
\DeclareMathOperator{\rad}{rad}
\DeclareMathOperator{\Lan}{Lan}
\DeclareMathOperator{\Ran}{Ran}
\DeclareMathOperator{\rank}{rank}
\DeclareMathOperator{\length}{length}
\DeclareMathOperator{\diam}{diam}

\begin{document}

\title{Topological spaces of persistence modules and their properties}
\author{Peter Bubenik}
\address{Department of Mathematics, University of Florida}
\email{peter.bubenik@ufl.edu}
\author{Tane Vergili}
\address{Department of Mathematics, Ege University}
\email{tane.vergili@ege.edu.tr}

\maketitle

\begin{abstract}
  Persistence modules are a central algebraic object arising in topological data analysis. The notion of interleaving provides a natural way to measure distances between persistence modules. We consider various classes of persistence modules, including many of those that have been previously studied, and describe the relationships between them. In the cases where these classes are sets, interleaving distance induces a topology. We undertake a systematic study the resulting topological spaces and their basic topological properties.
\end{abstract}
 

\section{Introduction}

A standard tool in topological data analysis is persistent homology~\cite{ghrist:survey,carlsson:topologyAndData,ghrist:survey2,Chazal:2017,Wasserman:2016}. It is often applied as follows. One starts with some data, constructs an increasing family of complexes or spaces, and applies homology with coefficients in some fixed field to obtain a persistence module. Next, one computes a summary (e.g. barcode~\cite{Collins2004881}, persistence diagram~\cite{cseh:stability}, or persistence landscape~\cite{bubenik:landscapes,bubenikDlotko}) which determines this persistence module up to isomorphism. In practice, one computes these summaries directly from the increasing family of complexes or spaces. Nevertheless, the persistence module is the central algebraic object in this pipeline, and has been a focus of research.

A key discovery in the study of persistence modules is the notion of \emph{interleaving}~\cite{ccsggo:2009} which provides a way of measuring the distance between persistence modules. For many persistence modules, this distance equals the \emph{bottleneck distance}~\cite{cseh:stability} between the corresponding persistence diagrams~\cite{Lesnick:2015,bubenikScott:1}. 
Interleavings and the resulting interleaving distance have been extensively studied both for the persistence modules considered here~\cite{Lesnick:2015,bubenikScott:1,Bauer:2015,Bauer:2016,cdsgo:book,bdsn,Blumberg:2017}, for Reeb graphs~\cite{deSilva:2016,Munch:2015}, for zig-zag persistence modules~\cite{Botnan:2016}, for multiparameter persistence modules~\cite{Lesnick:2015}, and for more general persistence modules~\cite{bdss:1,bdss:2,deSilva:2017,Bjerkevik:2017,Meehan:2017b,Meehan:2017a}.
     
For sets of persistence modules, the interleaving distance induces a topology. The main goal of the research reported here is to study the basic topological properties of the resulting topological spaces. 

Unfortunately, this research program runs into an immediate difficulty: the collection of persistence modules is not a set, but a proper class. While it is possible the consider this class with the interleaving distance~\cite{bdss:1,bdss:2,bdsn}, here we want to work with actual topological spaces.

So to start, we consider various classes of persistence modules. These include  classes that have been previously considered in theoretical work, such as pointwise finite-dimensional persistence modules~\cite{Crawley-Boevey:2015}, q-tame persistence modules~\cite{cdso:geometric}, interval-decomposable persistence modules, ephemeral persistence modules~\cite{ccbds}, and constructible persistence modules~\cite{Patel:2016a,curry:thesis}, as well as classes of persistence modules that arise in applications, such as those decomposable into finitely many interval modules, where each interval lies in some fixed bounded closed interval.  

We determine various relationships between these classes, such as inclusion (Figure~\ref{fig:sets-and-classes}). We also identify pairs of classes where for each element of one, there is an element of the other that has interleaving distance $0$ from the first (Section~\ref{sec:almost-inclusion}). We define and calculate an asymmetric distance we call \emph{enveloping distance} that measures how far one needs to expand a given class to include another (Section~\ref{sec:enveloping}). These two results are summarized in Figure~\ref{fig:enveloping-distances}. 

Next, we determine which of these classes are sets and which are proper classes. We show that the classes of interval-decomposable persistence modules and q-tame persistence modules are not sets (Corollary~\ref{cor:not-set}), though the classes of pointwise finite-dimensional persistence modules and persistence modules decomposable into countable-many interval modules are sets (Propositions \ref{prop:cid-set} and \ref{prop:pfd-set}). We introduce a set of persistence modules containing these two sets that consists of persistence modules decomposable into a set of interval modules with cardinality of the continuum (Definition~\ref{def:kid} and Proposition~\ref{prop:kid-set}).

For the remainder, we restrict ourselves to the identified sets of persistence modules and the topologies induced by the interleaving distance (Figure~\ref{fig:sets}). We identify which of the inclusions in Figure~\ref{fig:sets} are inclusions of open sets (Proposition~\ref{prop:open}).

We show that these topological spaces are large and poorly behaved in the following ways. They do not have the $T_0$ or Kolmogorov property (Corollary~\ref{cor:t0}), they are not locally compact (Corollary~\ref{cor:locally-compact}), and their topological dimension is infinite (Corollary~\ref{cor:top-dim}). In fact, we prove the following.

\begin{theorem}[Cube Theorem (Theorem~\ref{thm:embedding-cube})]
  Let $N \geq 1$. There exists an $\eps>0$ such that there is an isometric embedding of the cube $[0,\eps]^N$ with the $L^{\infty}$ distance into each of our topological spaces of persistence modules.
\end{theorem}

On the other hand, our topological spaces of persistence modules do have the following nice properties. They are paracompact (Lemma~\ref{lem:paracompact}), first countable (Lemma~\ref{lem:first-countable}), and are compactly generated (Lemma~\ref{lem:compactly-generated}).

We determine which of these topological spaces are separable (Theorems \ref{thm:separable} and \ref{thm:not-separable}), as well as second countable and Lindel\"of (Lemma~\ref{lem:tfae}).
We show that the space of pointwise finite-dimensional persistence modules is not complete (Theorem~\ref{thm:not-complete}), but that the space of persistence modules that are both q-tame and that decompose into countably-many intervals is complete (Theorem~\ref{thm:complete}).
We prove a Baire category theorem for complete extended pseudometric spaces (Theorem~\ref{thm:baire}) that implies that this space is also a Baire space (Corollary~\ref{cor:baire}).

We also identify the path components of the zero module in our topological spaces (Propositions \ref{prop:path-component-fid} and \ref{prop:path-component-cfid}), and show that they are contractible (Proposition~\ref{prop:contractible}).

Along the way, we observe the following mild strengthening of the structure theorem for persistent homology~\cite{ccbds}, which may be of independent interest.

\begin{theorem}[Structure Theorem (Theorem~\ref{thm:rad})]
  The radical of a q-tame persistence module is a countable direct sum of interval modules. 
\end{theorem}

\subsection*{Persistence modules and persistence diagrams}

Topological data analysis tends to focus on persistence diagrams~\cite{cseh:stability} rather than persistence modules. Readers more familiar with persistence diagrams may wonder why we work with persistence modules and what our results imply for persistence diagrams.

Let us present three responses.
First, persistent homology produces persistence modules. In many but not all cases, these persistence modules may be represented by a persistence diagram. Mathematically, persistence modules are the fundamental object of study. 
Second, one of our main motivations was to develop a theory that could be extended to multiparameter persistence modules~\cite{carlssonZomorodian:multidimP,Lesnick:2015} and generalized persistence modules~\cite{bdss:1,deSilva:2017,bdss:2}. In this more general setting there is no hope for an analog of the persistence diagram.
Third, our results for persistence modules may be used to obtain results for persistence diagrams as corollaries.

To be more precise, consider persistence modules that are
pointwise finite-dimensional (see Section~\ref{sec:class})
with the interleaving distance.
This forms an extended pseudometric space that we label \pfd.
If we take the quotient obtained by identifying persistence modules with zero interleaving distance, then we obtain an extended metric space that is isometric with a space of persistence diagrams with the bottleneck distance~\cite{cseh:stability}. This is the celebrated isometry theorem~\cite{ccsggo:interleaving,Lesnick:2015,cdsgo:book,Bauer:2015,bubenikScott:1}. 
Call this extended metric space \pd.

Now \pd\ inherits many of the properties of \pfd.
Specifically, 
it is not totally bounded,
any element of \pd\ does not have a compact neighborhood,
it is not path connected,
the path component of the empty persistence diagram consists of persistence diagrams without points with infinite persistence,
and this path component is contractible.
Furthermore, \pd\ is not separable and is not complete.
In addition, for each $N$ there is an $\eps>0$ such that there is an isometric embedding of the $N$-cube with diameter $\eps$  and the $L^{\infty}$ distance into \pd.
So the topological dimension of \pd\ is infinite.

\subsection*{For the data scientist}

For the reader primarily interested in topological data analysis, we would summarize our results by stating that the extended metric space of persistence diagrams with the bottleneck distance is ``big''. Say we fix $c<d$ and restrict ourselves to persistence diagrams with finitely many points $(a_i,b_i)$ each of which satisfies $c \leq a_i < b_i \leq d$.
This is a metric space. 
However, every neighborhood of every persistence diagram in this metric space 
is not compact.
Also, the topological dimension of this metric space is infinite.

In order to apply certain statistical and machine learning tools, one may be tempted to start with a compact set of persistence diagrams. In light of these results, this is a drastic step.

\subsection*{Extended pseudometric spaces}

The results presented here for extended pseudometric spaces are straight-forward extensions of the standard results for metric spaces 
(Lemmas \ref{lem:paracompact}, \ref{lem:first-countable}, \ref{lem:compactly-generated}, and \ref{lem:tfae} and Theorem~\ref{thm:baire}).
However, in order to keep the material accessible to applied mathematicians without a background in point-set topology, we include the proofs.

\subsection*{Related work}
 
Mileyko, Mukherjee, and Harer~\cite{mmh:probability} consider the set of persistence diagrams with countably many points in $\mathbb{R}^2$ together with the topology induced by the $p$-Wasserstein distance for $1 \leq p < \infty$. They show that the subspace consisting of persistence diagrams with finite distance to the empty persistence diagram is complete and separable.
We show the corresponding space for the bottleneck distance $(p=\infty)$ is complete (Theorem~\ref{thm:complete}) but not separable (Theorem~\ref{thm:not-separable}).
In a subsequent paper with Turner~\cite{tmmh:frechet-means} they study geometric properties of the same set with a slightly different metric. 
 
Blumberg, Gal, Mandell, and Pancia~\cite{Blumberg:2014} show that the set of persistence diagrams with finitely many points with the bottleneck distance is separable and that its Cauchy completion is separable. This completion is the set of persistence diagrams with the property that for every $\eps > 0$ there are only finitely many points with persistence at least $\eps$.
 
The authors have been informed of related work that is in preparation. Perea, Munch, and Khasawneh~\cite{Perea:2018} have characterized (pre)compact sets of persistence diagrams with the bottleneck distance. Their results imply that compact sets have empty interior. Cruz~\cite{Cruz:2018} has results on metric properties for generalized persistence diagrams with interleaving distance.


\subsection*{Organization of the paper}

In Section~\ref{sec:background}, we provide background on persistence modules, indecomposable modules, interleaving distance, and pseudometric spaces. In Section~\ref{sec:set-theory}, we define the classes of persistence modules that we consider, study the relationships between them, and identify which of them are sets. In Section~\ref{sec:top}, we study the basic topological properties of our topological spaces of persistence modules. 
Throughout, most of our arguments are elementary, except our proof of completeness which uses basic ideas from category theory.
We also provide an appendix where we examine interleavings of interval modules.

\section{Background} \label{sec:background}

In this section we define persistence modules and interleaving distance, giving examples and basic properties. We also define extended pseudometric spaces and their induced topological spaces.

\subsection{Persistence modules}

Let $\gf$ be a fixed field. 
A \emph{persistence module} $M$ is a set of 
$\gf$-vector spaces $\{M(a) \ | \ a \in \mathbb{R}\}$ together with $\gf$-linear maps $\{v_a^b: M(a) \to M(b) \ | \ a\leq b \}$ such that 
\begin{description}
	\item[i)] for all $a$, $v_a^a : M(a) \to M(a)$ is the identity map, and 
	\item[ii)] if $a \leq b \leq c$ then $v_a^c=v_b^c\circ v_a^b$.  \end{description}
Equivalently, a persistence module is a functor $M: \R \to \vect$, where $\R$ is the category whose set of objects is $\mathbb{R}$ and whose morphisms are the inequalities $a \leq b$, and $\vect$ is the category of $\gf$-vector spaces and $\gf$-linear maps.


\begin{example}
	Let $X$ be a topological space and $f: X \to \mathbb{R}$ be a function. For each $a\in \mathbb{R}$ the subset \[F_a:=\{x\in X \ | \ f(x)\leq a\} \subset X\]  is called a \emph{sublevel set}. Note that $a\leq b$ implies $F_a\subset F_b$ so that we have an inclusion map $i_a^b: F_a \xhookrightarrow{} F_b$ for all $a\leq b$. This inclusion map induces a linear map 
	\[H_n(i_a^b): H_n(F_a;\gf) \to H_n(F_b;\gf)\]
	on singular homology groups  with a coefficients in $\gf$ of degree  $n\geq 0$. We thus have a persistence module $HF: \R \to \vect$ given by  
$HF(a) = H_n(F_a;\gf)$ and $HF(a\leq b) =  H_n(i_a^b)$.
\end{example} 

\begin{example}
	Consider the half open interval $[0,2)$ in $\mathbb{R}$ and define the persistence module $\chi: \R \to \vect$ given by
\begin{equation*}
\chi(a) = \begin{cases}  \gf & a \in [0,2) \\ 0  &  \text{otherwise}   \end{cases} \quad \text{and} \quad
	\chi(a\leq b) = \begin{cases}  \Id & a,b \in [0,2) \\ 0  & \text{otherwise}   \end{cases} 
      \end{equation*}
      where $\Id$ is the identity map on $\gf$. For simplicity, we will abuse notation and denote this persistence module by $[0,2)$.
\end{example}

\begin{example}
Replacing $[0,2)$ in the above with an arbitrary interval $J \subset \mathbb{R}$ we obtain a persistence module that we call an \emph{interval module} and we will also denote by $J$.
\end{example}

\begin{example}
  A trivial but important example is the \emph{zero module}, denoted $\Zero$, that has $\Zero(a) = 0$ for all $a$.
\end{example}

A \emph{morphism of persistence modules} $M$ and $N$ is a collection of linear maps 
$\{\varphi_a: M(a) \to N(a) \ | \ a \in \mathbb{R}\}$ such that the following diagram commutes for each pair $a\leq b$.

\begin{equation}\label{cd:morphism}
\begin{tikzcd}[row sep=scriptsize]
  M(a) \ar[r, "v_a^b"] \ar[d,"\varphi_a",swap]
  & M(b) \ar[d, "\varphi_b"] \\
  N(a) \ar[r, "w_a^b"]
  & N(b) 
\end{tikzcd}
\end{equation}
Equivalently, a morphism of persistence modules is a natural transformation $\varphi:M \To N$.
We will often denote a morphism of persistence modules as $\varphi:M \to N$.
Such a morphism is an \emph{isomorphism} if and only if each linear map $\varphi_a$ is an isomorphism. 

\begin{example} \label{ex:morphism-intervals}
  It is a good exercise to check that because of the constraints due to the commutative squares in \eqref{cd:morphism}, there is a nonzero morphism from the interval module $[a,b)$ to the interval module $[c,d)$ only if $c \leq a \leq d \leq b$.
\end{example}

In the appendix, we present a more thorough discussion of interval modules (Section~\ref{sec:interval-relations}) and maps between them (Section~\ref{sec:nonz-maps-interv}).


\subsection{Indecomposables}

Given two persistence modules $M$ and $N$, their \emph{direct sum} is the persistence module $M \oplus N$ given by $(M \oplus N)(a) = M(a) \oplus N(a)$ and $(M \oplus N)(a \leq b) = M(a \leq b) \oplus N(a \leq b)$. In the same way we can define the direct sum of a collection of persistence modules indexed by an arbitrary set.


A persistence module is said to be \emph{indecomposable} if it is not isomorphic to a nontrivial direct sum.
For example, interval modules are indecomposable.
However, not all indecomposable persistence modules are interval modules (see \cite[Theorem 2.5, Remark 2.6]{cdsgo:book} for a discussion of examples due to do Webb~\cite{Webb:1985}, Lesnick, and Crawley-Boevey).




A special case of the following theorem follows from work of Gabriel~\cite{Gabriel:1972}, but the general case was proved by Crawley-Boevey~\cite{Crawley-Boevey:2015}.

\begin{theorem}[Structure Theorem] \label{thm:structure}
Let $M: \R \to \vect$ be a persistence module. If $M(a)$ is finite dimensional  for each $a\in \mathbb{R}$,  then $M$ is isomorphic to a direct sum of interval modules.
\end{theorem}


\subsection{Interleaving distance} \label{sec:interleaving}

Interleaving distance was introduced in~\cite{ccsggo:2009} and further studied in the context of multiparameter persistence in~\cite{Lesnick:2015}. Here we also adopt the categorical point of view from~\cite{bubenikScott:1}.


\begin{definition} \label{def:interleaving}
Let $\eps \geq 0$. 
An $\eps$-interleaving between persistence modules $M$ and $N$ consists of morphisms
$\varphi_a: M(a) \to N(a + \eps)$ and $\psi_a:N(a)\to M(a+\eps)$ for all $a$ 
such that the following four diagrams commute for all $a \leq b$, where the horizontal maps are given by the respective persistence modules. 
\begin{equation} \label{cd:parallelogram}
\begin{tikzcd}[row sep=scriptsize]
M(a) \ar[r] \ar[rd, "\varphi_a",swap] 
& M(b) \ar[rd, "\varphi_b"] \\
& N(a+\eps) \ar[r] 
& N(b+\eps)
\end{tikzcd}\hspace{30pt}%
\begin{tikzcd}[row sep=scriptsize]
& M(a+\eps) \ar[r] \ar[r]  
& M(b+\eps)   \\
N(a) \ar[r] \ar[ru, "\psi_a"] 
& N(b) \ar[ru, "\psi_b", swap] 
&
\end{tikzcd}
\end{equation} \\
\begin{equation} \label{cd:triangle}
\begin{tikzcd}[row sep=scriptsize]
M(a)  \ar[rr]  \ar[rd, "\varphi_a", swap] 
& 
&  M(a+2\eps) \\
& N(a+\eps) \ar[ru, "\psi_{a+\eps}", swap]  
& 
\end{tikzcd}\hspace{30pt}%
\begin{tikzcd}[row sep=scriptsize]
& 
M(a+\eps)  \ar[rd, 	"\varphi_{a+\eps}"]
&    \\
N(a)   \ar[ru, "\psi_{a}"] \ar[rr]  
& 
& N(a+2\eps) 
\end{tikzcd}
\end{equation}
Equivalently, we may describe this in terms of natural transformations.
First, for $x \in \mathbb{R}$ let $T_{x}: \R \to \R$ denote the functor given by $T_x(a) = a + x$.
Next if $x \geq 0$, let $\eta_x: \Id_{\R} \To T_{x}$ denote the natural transformation from the identity functor on $\R$ to $T_x$ that has components $(\eta_x)_a: a \leq a+x$.
Then an $\eps$-interleaving consists of natural transformations $\varphi: M \To N T_{\eps}$ and $\psi: N \To M T_{\eps}$ such that
$(\psi T_{\eps}) \varphi = M \eta_{2\eps}$ and $(\varphi T_{\eps}) \psi = N \eta_{2\eps}$.
See~\cite[Section 3]{bubenikScott:1} for more details.
We say $M$ and $N$ are \emph{$\eps$-interleaved}.
\end{definition}

\begin{remark}  \label{int} Two persistence modules are $0$-interleaved if and only if they are isomorphic.
  If persistence modules $M$ and $N$ are $\eps$-interleaved and $N$ and $P$ are $\delta$-interleaved then $M$ and $P$ are $(\eps+\delta)$-interleaved.
\end{remark}

\begin{definition}
	Let $M$ and $N$ be two persistence modules. Then the \emph{interleaving distance} $d_I(M,N)$ between $M$ and $N$ is defined as 
	\[d_I(M,N):= \inf\big(\eps \in [0,\infty) \ | \ M \ \text{and} \ N  \ \text{are} \  \ \eps\text{-interleaved} \ \big)  \] If no such $\eps$ exists, then $d_I(M,N)=\infty$. 
\end{definition}

\begin{example} \label{intzero}
  The interval modules $[0,2]$ and $(0,2)$ are not $0$-interleaved. In fact, there are no nonzero maps between $[0,2]$ and $(0,2)$. 
However they are $\eps$-interleaved for all $\eps>0$. Thus, $d_I([0,2],(0,2)) = 0$.
%
\end{example}

\begin{example} 
  The interval modules $M=[0,1)$ and $N=[0,\infty)$ are not $\eps$-interleaved for any $\eps \geq 0$. Indeed, assume $\varphi$ and $\psi$ provide such an interleaving. 
Consider the following trapezoid. 
\begin{equation*}
\begin{tikzcd}[row sep=scriptsize]
& M(\eps) \ar[rr, "M(\eps \leq 2+\eps)"] & & M(2+\eps) \ar[dr, "\varphi_{2+\eps}"] \\
N(0) \ar[ru, "\psi_0"]  \ar[rrrr, "N(0\leq2+2\eps)"]
& & & & N(2+2\eps)
\end{tikzcd}
\end{equation*}
It decomposes into a commutative parallelogram and commutative triangle from \eqref{cd:parallelogram} and \eqref{cd:triangle} in two different ways. In either case, this diagram commutes.
Furthermore, the bottom horizontal arrow is the identity on $\gf$ and the top horizontal arrow is $0$, which is a contradiction.
\end{example}

In the appendix, we give a careful study of interleavings of interval modules (Section~\ref{sec:interl-interv-modul}).

We will make use of the following lemma without reference.

\begin{lemma}[Converse Algebraic Stability Theorem~{\cite[Theorem 3.4]{Lesnick:2015}}]
 Let $\eps \geq 0$. If for all $\alpha \in A$, the persistence modules $I_{\alpha}$ and $J_{\alpha}$ are $\eps$-interleaved, then $\bigoplus_{\alpha \in A} I_{\alpha}$ and $\bigoplus_{\alpha \in A} J_{\alpha}$ are $\eps$-interleaved. 
Thus $d_I(\bigoplus_{\alpha \in A} I_{\alpha},\bigoplus_{\alpha \in A} J_{\alpha}) \leq \sup_{\alpha \in A}d_I(I_{\alpha},J_{\alpha})$.
\end{lemma}

\begin{proof}
  For $\alpha \in A$, let $\varphi_{\alpha}$ and $\psi_{\alpha}$ be maps giving an $\eps$-interleaving of $I_{\alpha}$ and $J_{\alpha}$. Then $\bigoplus \varphi_{\alpha}$ and $\bigoplus \psi_{\alpha}$ provide the desired $\eps$-interleaving.
\end{proof}

\subsection{Pseudometric spaces}
\label{sec:pseudometric}

\begin{definition} 
	A \emph{pseudometric} on a set $X$ is a map $d: X\times X \to [0,\infty)$ that satisfies 
	\begin{description}   
		\item[M1)] $d(x,x)=0$,
		\item[M2)] $d(x,y)=d(y,x)$, and 
		\item[M3)] $d(x,y)\leq d(x,z)+d(z,y)$ 
	\end{description}	for all $x,y,z \in X$.  
Note that we have omitted the condition $d(x,y)=0$ implies $x=y$ required of metric.
More generally, an \emph{extended pseudometric} on $X$ is a map $d: X \times X \to [0,\infty]$ satisfying the same three axioms.
We call a set with an (extended) pseudometric an \emph{(extended) pseudometric space}.
\end{definition}

\begin{theorem}[\cite{ccsggo:2009,Lesnick:2015,bubenikScott:1}]
  The interleaving distance is an extended pseudometric on any set of (isomorphism classes of) persistence modules.
\end{theorem}

\begin{remark}
  A proper class of persistence modules with the interleaving distance is not an extended pseudometric space since it is not a set. However it is a symmetric Lawvere space~\cite{bdss:1,bdsn,bdss:2}. 
\end{remark}

In an extended (pseudo)metric space, the condition $d(x,y)<\infty$ defines an equivalence relation. As a result, such a space has a natural partition into (pseudo)metric spaces. 

In an (extended) pseudometric space one can consider equivalence classes of the equivalence relation $x \sim y$ if $d(x,y)=0$ to obtain an (extended) metric space. However, for persistence modules, one may be interested in distinguishing nonisomorphic modules with zero interleaving distance, so we will not apply this simplification.

Any extended pseudometric on a set induces a \emph{topology} on it. Indeed, for any $x\in X$ and a real number $r>0$ consider the \emph{open ball} $B_r(x)$ centered at $x$ with radius $r$,
\[B_r(x):=\{y\in X \ | \ d(x,y)<r \}.\]  
We call a set $O$  \emph{open} in $X$ if for each $x\in O$, there exists $r>0$ such that $B_r(x)\subset O$. Then it is easy to check that the collection of all open sets is a topology on $X$.

Note that each open ball $B_r(x)$ is also an open set in $X$ and the collection of all open balls forms a base for this topology $X$ since each open set $O$ in $X$ can be written as a union of open balls.

\begin{example}
  Consider the interval module $[0,5)$ and let $\eps >1$. Then the ball 
$B_{\eps}([0,5))$
contains the interval modules $[-1,6]$ and $(1,4)$.
\end{example}

In the appendix, we study the interval modules in an $\eps$-neighborhood of an interval module (Section~\ref{sec:neighb-interv-modul}).

A sequence $(x_n)_{n \geq 1}$ in an extended pseudometric space $X$ is said to  \emph{converge} to $x \in X$ if for all $\eps > 0$ there exists $N>0$ such that for all $n \geq N$, $d(x_n,x) < \eps$. The point $x$ is called a \emph{limit} of the sequence. Note that in an extended pseudometric space we no longer have unique limits, but we do have that if $x$ and $x'$ are limits, then by the triangle inequality $d(x,x') = 0$.

A sequence $(x_n)_{n \geq 1}$ in an extended pseudometric space is a
\emph{Cauchy sequence} if for all $\eps>0$ there exists an $N>0$ such that for all $n,m \geq N$, $d(x_n,x_m)< \eps$.
If a subsequence of a Cauchy sequence has a limit $x$, then by the triangle inequality, $x$ is also a limit of the Cauchy sequence.

\section{Sets and classes of persistence modules}
\label{sec:set-theory}

In this section we define classes of persistence modules that contain many of the persistence modules considered in the literature. We study the relationships between these classes and determine which of them are in fact sets.

For the remainder of the paper, we will only consider isomorphism classes of persistence modules. That is, whenever we say `persistence module', we really mean `isomorphism class of persistence modules'. This is standard when discussing both vector spaces and persistence modules.

\subsection{Classes of persistence modules}
\label{sec:class}

In this section, we consider the classes of persistence modules in Figure~\ref{fig:sets-and-classes}, which we now describe.

\begin{figure}
	\begin{tikzpicture} [
	auto,
	decision/.style = { diamond, draw=blue, thick, fill=blue!10,
		text width=4em, text badly centered,
		inner sep=1pt, rounded corners },
	block/.style    = { rectangle, draw=blue, thick, 
		fill=blue!9, text width=4em, text centered,
		rounded corners, minimum height=2em },
	line/.style     = { draw, thick, ->, shorten >=2pt },
	]
	\matrix [column sep=15mm, row sep=8mm] {
		&  &  & \node[block] (pers) {$\pers$};  &  &  \\
		&  &  \node [block] (id) {$\id$}; &  &  \node [block] (q-tame) {$\qtame$};  & \\
		&   \node [block] (cid) {$\cid$}; &  &  \node [block] (pfd) {$\pfd$};  &  &  \\
		\node [block] (cfid) {$\cfid$}; & & \node [block] (fid) {$\fid$}; & &  &  &\\
		&  \node [block] (ffid) {$\ffid$}; &  &  &  &  \\
		& \node [block] (ffids) {$\ffids$}; &  & \node [block] (eph) {$\eph$}; &  & \\
		&  & \node [block] (zm) {$\zm$};  &  &  & \\
	};
	\begin{scope} [every path/.style=line]
	\path (zm) -- (ffids);
	\path (zm) -- (eph);
	\path (eph) -- (q-tame);
	\path (eph) -- (id);
	\path (ffids) -- (ffid);
	\path (ffid) -- (fid);
	\path (ffid) -- (cfid);
	\path (cfid) -- (cid);
	\path (fid) -- (cid);
	\path (fid) -- (pfd);
	\path (cid) -- (id);
	\path (pfd) -- (id);
	\path (pfd) -- (q-tame);
	\path (id) -- (pers);
	\path (q-tame) -- (pers);
	\end{scope}


\end{tikzpicture}
\caption{Hasse diagram of sets and classes of persistence modules.}
\label{fig:sets-and-classes}
\end{figure}
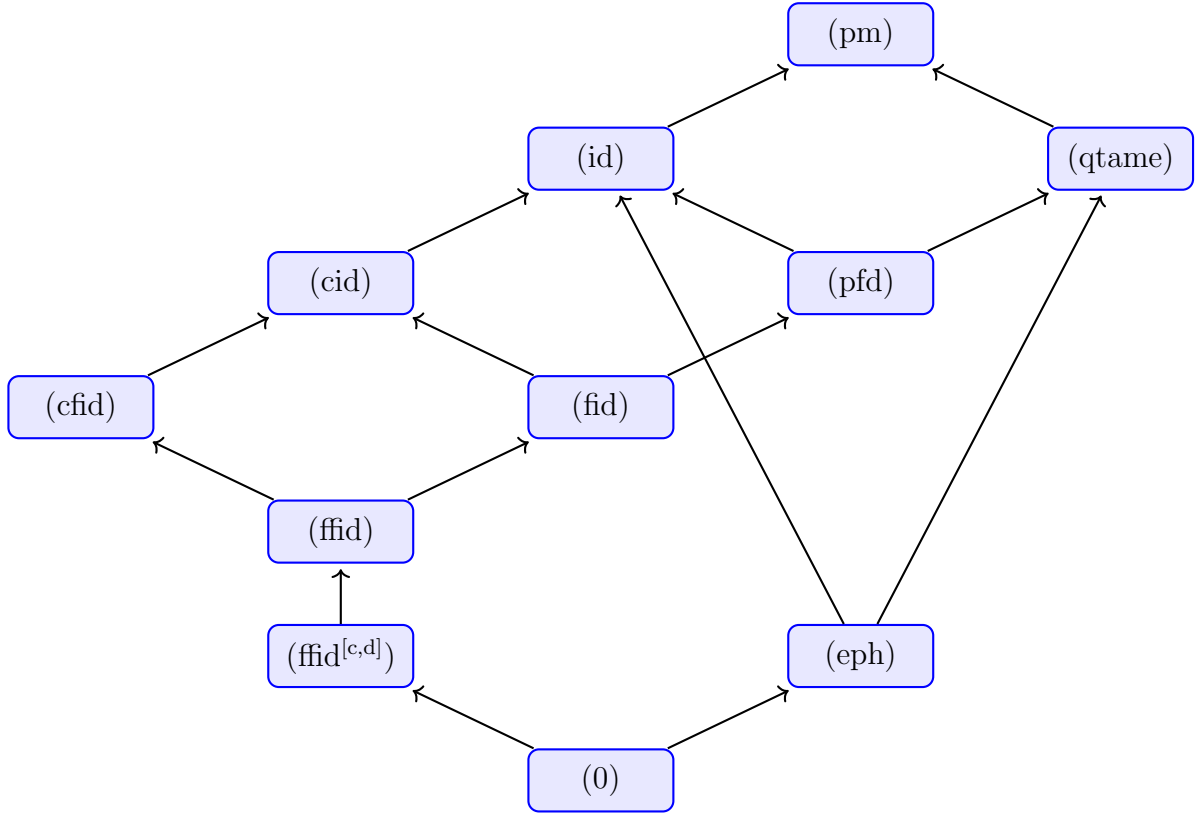


\begin{itemize}  
\item $\pers$ is the class of persistence modules.
\item $\id$ is the class of \emph{interval-decomposable} persistence modules: those isomorphic to $\bigoplus_{\alpha \in A} I_{\alpha}$, where $A$ is some indexing set, and each $I_{\alpha}$ is an interval module.
\item $\cid$, the \emph{countably interval-decomposable} persistence modules, is the subclass of $\id$ where the index set $A$ is countable.
\item $\cfid$, the \emph{countably finite-interval decomposable} persistence modules, is the subclass of $\cid$ in which each interval $I_{\alpha}$ is finite.
\item $\fid$, the \emph{finitely interval-decomposable} persistence modules, is the class of persistence modules isomorphic to 
$\bigoplus_{k=1}^N I_k$ for some $N$, where each $I_k$ is an interval module.
\item $\ffid$, the \emph{finitely finite-interval decomposable} persistence modules, is the subclass of $\fid$ in which each $I_k$ is a finite interval.
\item Given $c<d$, $\ffids$ is the subclass of  $\ffid$ in which each $I_k \subset [c,d]$.
\item $\pfd$, the \emph{pointwise finite dimensional} persistence modules, is the class of all persistence modules $M$ with each $M(a)$  finite dimensional.
\item $\qtame$, the \emph{q-tame} persistence modules, is the class of all persistence modules $M$ where each $a<b$ the linear map $v_a^b: M(a) \to M(b)$  has a finite rank.
\item $\eph$, the \emph{ephemeral} persistence modules, is the class of all persistence modules $M$ where for each $a<b$ the linear map
$v_a^b: M(a) \to M(b)$ is zero. 
\item $\zm$ is the class consisting of only the zero persistence module.
\end{itemize}

\begin{remark}
  The class $\fid$ is a slight generalization of the class of \emph{constructible} persistence modules.
  A persistence module $M$ is said to be \emph{constructible}~\cite{Patel:2016a} if there exists a finite subset $A=\{a_1,\ldots,a_n\}$ of $\mathbb{R}$ such that
  \begin{itemize}
  \item for $t< a_1$, $M(t)=0$, 
  \item for $a_i \leq s \leq t < a_{i+1}$, $M(s\leq t)$ is an isomorphism
  where $i\in \{1,\ldots,n-1\}$ , and
  \item for $a_n\leq s\leq t$, $M(s\leq t)$ is an isomorphism. 
  \end{itemize}
A constructible module $M$, satisfies $M \isom \bigoplus_{k=1}^N I_k$ where each $I_k$ is of the form $[a_i,a_j)$ or $[a_i,\infty)$.\footnote{In particular, the multiplicity of $[a_i,a_j)$ can be calculated 
using the inclusion/exclusion formula $\rank M(a_i \leq a_{j-1}) - \rank M(a_i \leq a_j) - \rank M(a_{i-1} \leq a_{j-1}) + \rank M(a_{i-1} \leq a_j)$~\cite{cseh:stability}, which is an example of M\"obius inversion~\cite{Patel:2016a}.} 
\end{remark}


\subsection{Inclusions}

\begin{lemma} \label{lem:eph-decomposition}
  Let $M$ be an ephemeral module. Then $M \isom \bigoplus_{\alpha \in A} M_{\alpha}$, where each $M_{\alpha} \isom [r,r]$ for some $r \in \mathbb{R}$.
\end{lemma}

\begin{proof}
  Let $M \in \eph$. 
  For $r \in \mathbb{R}$, let $M_r$ be the persistence module with $M_r(x) = M(r)$ if $x=r$ and otherwise $M_r(x)=0$.
Then $M \isom \oplus_{r \in \mathbb{R}} M_r$.
Furthermore each $M(r)$ has a basis, so $M_r$ decomposes over this basis into $[r,r]$ interval modules.
\end{proof}

\begin{proposition}
  The diagram in Figure~\ref{fig:sets-and-classes} is a Hasse diagram for the poset structure of these classes of persistence modules under the inclusion order.
\end{proposition}

\begin{proof}
  By Theorem~\ref{thm:structure}, $\pfd$ is in $\id$.
  By Lemma~\ref{lem:eph-decomposition}, $\eph \subset \id$.
  It is easy to check that all of the other arrows indicated in the diagram are inclusions and that in fact all of the inclusions are proper. With the observation that if $A \subset B$, $C \subset D$ and $A \not\subset D$ then $B \not\subset C$, it remains to check the following cases.
  \begin{enumerate}
  \item $\eph \not\subset \pfd$: $\bigoplus_{k=1}^{\infty} [0,0]$ is in $\eph$ but not in $\pfd$.
  \item $\eph \not\subset \cid$: $\bigoplus_{r \in \mathbb{R}} [0,0]$ is in $\eph$ but not in $\cid$.
  \item $\ffids \not\subset \eph$: $[c,d]$ is in \ffids \ but is not in \eph.
  \item $\fid \not\subset \cfid$: $[0,\infty)$ is in \fid \ but is not in  \cfid. 
  \item $\pfd \not\subset \cid$: $\bigoplus_{r \in \mathbb{R}} [r,r]$ is in \pfd \ but is not in \cid. 
  \item $\cfid \not\subset \qtame$: $\bigoplus_{k=1}^{\infty}[0,1)$ is in \cfid \ but is not in \qtame.
  \item $\qtame \not\subset \id$: $\prod_{k=1}^{\infty}[0,\frac{1}{k})$ is in \qtame \ but is not in \id~\cite{ccbds}.  
  \end{enumerate}
\end{proof}

	




\subsection{Almost inclusions} \label{sec:almost-inclusion}

\begin{definition}
  Say that a class of persistence modules $\A$ \emph{almost includes} in a class of persistence modules $\B$ if for each $A \in \A$ there exists an element $B \in \B$ such that $d_I(A,B) = 0$.
\end{definition}

\begin{lemma} \label{lem:almost-inclusion}
  A finite sequence of inclusions and almost inclusions is an almost inclusion.
\end{lemma}

\begin{proof}
  This follows from the triangle inequality.
\end{proof}

\begin{lemma} \label{lem:eph-almost-inclusion} 
$M$ is an ephemeral persistence module if and only if $d_I(M,\Zero)=0$. That is, $\eph$ almost includes in $\zm$.
\end{lemma}

\begin{proof} 
Let $M$ be an ephemeral persistence module. Then $M$ and $\Zero$ are $\eps$-interleaved for all $\eps>0$ 
by the zero maps. 

Next assume $d_I(M,\Zero)=0$. Consider $a<b$. Let $\eps = \frac{b-a}{2}$.
Since $M$ and $\Zero$ are $\eps$-interleaved, the map $M(a<b)$ factors through 0, and is thus the zero map.
\begin{equation*}
\begin{tikzcd}[row sep=small]
M(a)  \ar[rr,"M(a<b)"]  \ar[rd, "\varphi_a", swap] & &  M(b) \\
& 0 \ar[ru, "\psi_{\frac{a+b}{2}}", swap]  & 
\end{tikzcd}
\end{equation*}
Therefore $M$ is an ephemeral persistence module.
\end{proof}



For a persistence module $M$, define the \emph{radical} of $M$ by $(\rad M)(a) = \sum_{c < a} \im M(c<a)$~\cite{ccbds}.
Note that $\rad M \subset M$ and inherits the structure of a persistence module.

\begin{proposition} \label{prop:rad}
  Let $M$ be a persistence module. Then $d_I(M,\rad M) = 0$.
\end{proposition}

\begin{proof}
  Let $\eps >0$. For all $a \in \mathbb{R}$,
  let $\varphi_a = M(a<a+\eps): (\rad M)(a) \to M(a+\eps)$, and
  let $\psi_a = M(a<a+\eps): M(a) \to (\rad M)(a+\eps)$.
  Then by the functoriality of $M$, this is an $\eps$-interleaving of $\rad M$ and $M$.
  Therefore $d_I(\rad M, M) = 0$.
\end{proof}

\begin{theorem} \label{thm:rad}
  Let $M \in \qtame$. Then $\rad M \in \qtame$ and $\rad M \in \cid$.
\end{theorem}

\begin{proof}
  Let $M \in \qtame$. Since $\rad M$ is a submodule of $M$, it follows that $\rad M \in \qtame$ as well. 
By \cite[Corollary 3.6]{ccbds}, $\rad M \in \id$. We will strengthen this to show that $\rad M \in \cid$.

Since $\rad M \in \id$, $\rad M \isom \bigoplus_{\alpha \in A} I_{\alpha}$.
For $q,r \in \mathbb{Q}$ with $q<r$, let $A_{q,r} = \{ \alpha \in A \st q,r \in I_{\alpha}\}$, and let $A' = \bigcup_{q<r \in \mathbb{Q}} A_{q,r}$.
Since $\rad M \in \qtame$, for each $q<r \in \mathbb{Q}$, $\abs{A_{q,r}} < \infty$.
Therefore $A'$ is countable.

Furthermore, by definition, for each $a \in \mathbb{R}$ and for each $x \in (\rad M)(a)$ there exists $c<a$ and $y \in M(c)$ such that $M(c\leq a)(y) = x$.
Choose $b \in (c,a)$. Then $z:= M(c\leq b)(y) \in (\rad M)(b)$ and $(\rad M)(b\leq a)(z) = x$.
Hence the interval decomposition of $\rad M$ does not contain any one-point intervals, and thus $A = A'$.
Therefore $\rad M \in \cid$. 
\end{proof}

Combining the previous two results we have the following.

\begin{corollary}
  Let $M \in \qtame$. Then there exists $N \in \cid$ such that $d_I(M,N) = 0$. That is, $\qtame$ almost includes in $\cid$.
\end{corollary}

\subsection{Enveloping distance} \label{sec:enveloping}

In this section, we define a non-symmetric distance between classes of persistence modules and calculate its value for most of the pairs in Figure~\ref{fig:sets-and-classes}.

\begin{definition} \label{def:ed}
  Let $\A$ and $\B$ be classes of persistence modules.
  We define the \emph{enveloping distance} from $\A$ to $\B$ as follows.
  \begin{equation*}
    E(\A,\B) = \inf (r \st \forall B \in \B \text{ and } s > r, \exists A \in \A\text{ such that } A,B \text{ are $s$-interleaved})
  \end{equation*}
If there is no such $r$, we set $E(\A,\B) = \infty$.
\end{definition}

For example, as we will demonstrate later in this section, $E(\zm,\ffids) = \frac{d-c}{2}$ and $E(\ffids,\zm) = 0$.

We will use the following basic fact about interleavings.

\begin{lemma}[\cite{Lesnick:2015,bubenikScott:1}] \label{lem:interleaving-additivity}
  If persistence modules $A$ and $B$ are $s$-interleaved and persistence modules $B$ and $C$ are $t$-interleaved, then $A$ and $C$ are $(s+t)$-interleaved.
\end{lemma}

The enveloping distance has the following properties.

\begin{lemma} \label{lem:ed}
  $E(\A,\A) = 0$ and $E(\A,\C) \leq E(\A,\B) + E(\B,\C)$.
\end{lemma}

\begin{proof}
  For reflexivity, each persistence module is $s$-interleaved with itself for all $s \geq 0$. The triangle inequality follows from Lemma~\ref{lem:interleaving-additivity}.
\end{proof}

\begin{definition} \label{def:red}
In the case that $E(\A,\B) = \infty$, we write that $E(\A,\B) = \infty^{-}$ 
if $\forall B \in \B \ \exists s$ and $A \in \A$ such that $A,B$ are $s$-interleaved. From now on we reserve $E(\A,\B)=\infty$ for the case that this condition is not satisfied.
\end{definition}

\begin{lemma} \label{lem:almost-inclusion-env}
  If $\A$ (almost) includes in $\B$ then $E(\B,\A) = 0$. 
\end{lemma}

\begin{proof}
  This follows immediately from the definitions.
\end{proof}

\begin{corollary} \label{cor:env-zm-eph}
  $E(\zm,\eph) = 0$ and $E(\eph,\zm) = 0$.
\end{corollary}

\begin{lemma} \label{lem:enveloping-distance}
  If $\A$ (almost) includes in $\B$, $E(\B,\C) = \infty$, and $\C$ (almost) includes in $\D$, then $E(\A,\D) = \infty$.
\end{lemma}

\begin{proof}
  Assume $E(\A,\D) < \infty$.
  Then there is some $s \geq 0$ such that for all $D \in \D$ there exists an $A \in \A$ such that $D$ and $A$ are $s$-interleaved.

  Let $\eps>0$.
  Let $C \in \C$.
  Since $\C$ (almost) includes in $\D$, there is a $D \in \D$ such that $C$ and $D$ are $\eps$-interleaved.
  By our first observation, there is an $A \in \A$ such that $D$ and $A$ are $s$-interleaved.
  Since $\A$ (almost) includes in $\B$, there is a $B \in \B$ such that $A$ and $B$ are $\eps$-interleaved.
  Therefore by Remark~\ref{int}, $C$ and $B$ are $(s+2\eps)$-interleaved.
  So for all $C \in \C$ there is a $B \in B$ such that $C$ and $B$ are $(s+2\eps)$-interleaved.
  Thus $E(\B,\C)<\infty$.
\end{proof}

\begin{proposition} \label{prop:env}
  \begin{enumerate}
  \item We have the following enveloping distances: $E(\zm,\ffids) = \frac{d-c}{2}$ and  $E(\ffids,\ffid) = \infty^-$.
    Also, $E(\zm,\ffid) = \infty^-$, $E(\eph,\ffid) = \infty^-$ and $E(\eph,\ffids) = \frac{d-c}{2}$. 
  \item In addition, 
$E(\cfid,\fid)=\infty$ and  $E(\qtame,\cfid)=\infty$.
  \item With the exception of $\zm \subset \eph$, $\zm \subset \ffids$, $\ffids \subset \ffid$ and the possible exception of $\pfd \subset \qtame$, all of the other inclusions $\A \subset \B$ in Figure~\ref{fig:sets-and-classes} have enveloping distance $E(\A,\B)=\infty$. Also $E(\qtame,\cid) = \infty$.
  \end{enumerate}
\end{proposition}

\begin{proof}
  \begin{enumerate}
  \item 
    \begin{itemize}
    \item $\zm \subset \ffids$: $d_I([c,d],0) = \frac{d-c}{2}$ and for all $M \in \ffids$, $d_I(M,0) \leq \frac{d-c}{2}$.
    \item $\ffids \subset \ffid$: For all $M \in \ffids$ and $N \in \ffid$, $d_I(M,N) \leq d_I(M,0) + d_I(0,N) < \infty$. Let $z\geq 0$. For all $M \in \ffids$, there are no nontrivial maps from $M$ to $(d,d+2z]$. Thus $d_I(M,(d,d+2z]) \geq d_I((d,d+2z],0) \geq z$.
    \item The other three cases follow from the same arguments.
    \end{itemize}
\item
  \begin{itemize}
  \item $\cfid$ to $\fid$: Consider $[0,\infty)$.
  \item $\qtame$ to $\cfid$: Consider $\bigoplus_{k=1}^{\infty} [0,k)$. 
  \end{itemize}
  \item 
  \begin{itemize}
  \item $\ffid \subset \fid$: Consider $[0,\infty)$.
  \item $\ffid \subset \cfid$: Consider $\bigoplus_{k=1}^{\infty} [0,k)$.
  \item $\cfid \subset \cid$: Consider $[0,\infty)$.
  \item $\fid \subset \cid$: Consider $\bigoplus_{k=1}^{\infty} [0,\infty)$.
  \item $\fid \subset \pfd$: Consider $\bigoplus_{k=0}^{\infty} [2^k,2^{k+1})$.
  \item $\cid \subset \id$: Consider $\bigoplus_{r \in \mathbb{R}} [0,\infty)$.
  \item $\id \subset \pers$: Consider $\prod_{k=1}^{\infty} [0,\infty)$. 
  \item $\eph \subset \id$, $\eph \subset \qtame$, $\pfd \subset \id$, $\qtame \subset \cid$, and $\qtame \subset \pers$ follow from Lemma~\ref{lem:enveloping-distance}. 
  \end{itemize}
\end{enumerate}
\end{proof}

\begin{remark}
  Together with Corollary~\ref{cor:env-zm-eph}, Lemma~\ref{lem:almost-inclusion-env}, and
Lemma~\ref{lem:enveloping-distance}, 
this proposition implies all of the pairwise enveloping distances between the sets and classes of persistence modules in Figure~\ref{fig:enveloping-distances}, except
  $E(\pfd,\qtame)$.
For example, $E(\id,\qtame)=0$, $E(\cid,\pfd)=0$, and $E(\cid,\qtame)=0$ by Lemmas \ref{lem:almost-inclusion-env} and \ref{lem:almost-inclusion}, and
$E(\fid,\cfid)=\infty$ by Lemma~\ref{lem:enveloping-distance}. 
\end{remark}

We end this section by showing that $E(\pfd,\qtame)=0$. First we give a definition.

\begin{definition}
  Let $M$ be a persistence module. Let $p\geq 0$. We define the \emph{$p$-persistent submodule} of $M$ by
  \begin{equation*}
    M^{(p)}(a) = \im M(a-p \leq a).
  \end{equation*}
For $a \leq b$, there is an induced map between objects $M^{(p)}(a)$ and $M^{(p)}(b)$ given by $M(a \leq b)$.
Since $M$ is a persistence module, so is $M^{(p)}$, and since $M^{(p)}(a)$ is a sub-vector space of $M(a)$ for all $a$, $M^{(p)}$ is a submodule of $M$.
\end{definition}

\begin{proposition} \label{prop:Mp}
  Let $M$ be a persistence module and let $p\geq 0$. Then $M$ and $M^{(p)}$ are $p$-interleaved.
\end{proposition}

\begin{proof}
  For $a \in \R$, define $\varphi_a:M(a) \to M^{(p)}(a+p)$ by $\varphi_a = M(a\leq a+p)$, and $\psi_a:M^{(p)}(a) \to M(a+p)$ by $\psi_a = M(a \leq a+p)$.
Then all the arrows in diagrams \eqref{cd:parallelogram} and \eqref{cd:triangle} are maps in $M$ and hence commute.
\end{proof}

\begin{corollary}
  $E(\pfd,\qtame)=0$.
\end{corollary}

\begin{proof}
  Let $M$ be a q-tame persistence module. Let $p>0$. 
Then by definition, $M^{(p)}$ is a pointwise finite-dimensional persistence module.
By Proposition~\ref{prop:Mp}, $M$ and $M^{(p)}$ are $p$-interleaved.
Thus, by definition, $E(\pfd,\qtame)=0$.
\end{proof}

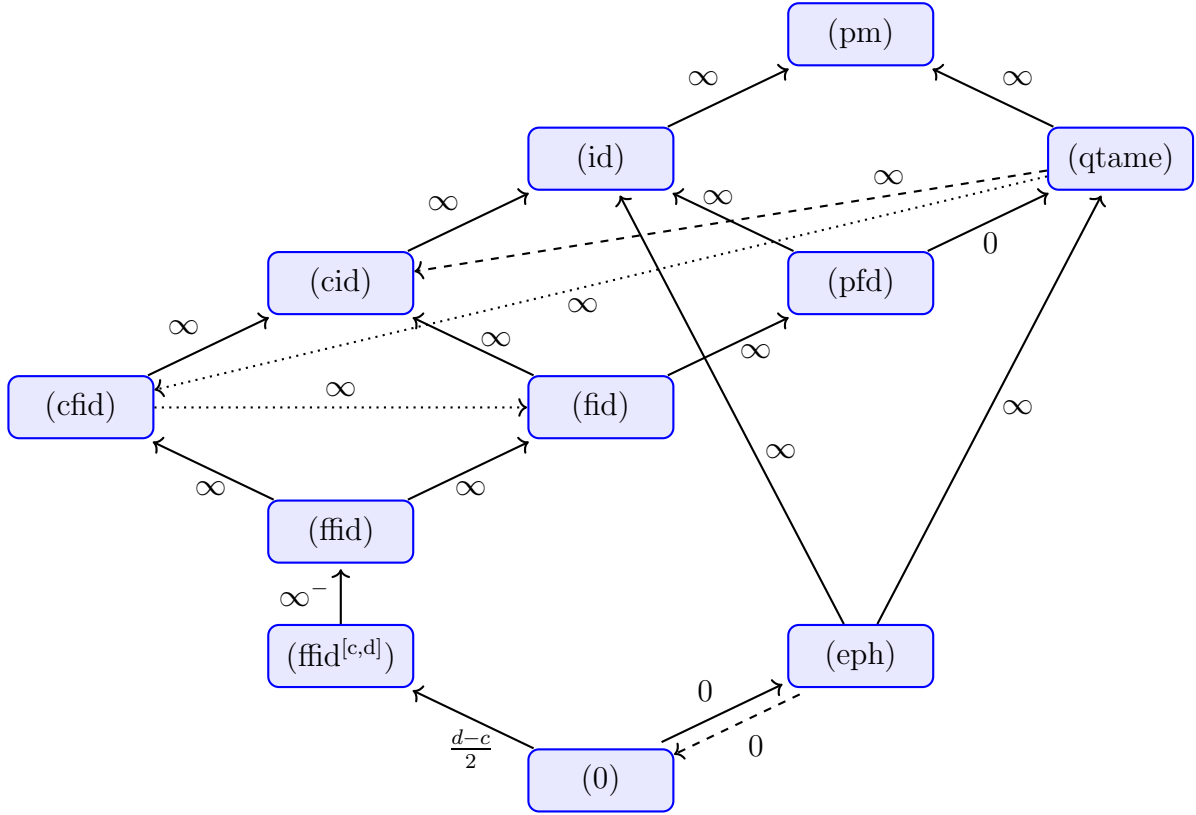
\begin{figure}
	\begin{tikzpicture} [
	auto,
	decision/.style = { diamond, draw=blue, thick, fill=blue!10,
		text width=4em, text badly centered,
		inner sep=1pt, rounded corners },
	block/.style    = { rectangle, draw=blue, thick, 
		fill=blue!9, text width=4em, text centered,
		rounded corners, minimum height=2em },
	line/.style     = { draw, thick, ->, shorten >=2pt },
	]
	\matrix [column sep=15mm, row sep=8mm] {
		&  &  & \node[block] (pers) {$\pers$};  &  &  \\
		&  &  \node [block] (id) {$\id$}; &  &  \node [block] (q-tame) {$\qtame$};  & \\
		&   \node [block] (cid) {$\cid$}; &  &  \node [block] (pfd) {$\pfd$};  &  &  \\
		\node [block] (cfid) {$\cfid$}; & & \node [block] (fid) {$\fid$}; & &  &  &\\
		&  \node [block] (ffid) {$\ffid$}; &  &  &  &  \\
		& \node [block] (ffids) {$\ffids$}; &  & \node [block] (eph) {$\eph$}; &  & \\
		&  & \node [block] (zm) {$\zm$};  &  &  & \\
	};
	\begin{scope} [every path/.style=line]
	\path (zm) -- node[below]{$\frac{d-c}{2}$} (ffids);
	\path [transform canvas={xshift=-.2em},transform canvas={yshift=.2em}] (zm) -- node{$0$} (eph);
	\path (eph) -- node[right]{$\infty$} (q-tame);
	\path (eph) -- node[pos=.4,right]{$\infty$} (id);
	\path (ffids) -- node{$\infty^-$} (ffid);
	\path (ffid) -- node[below]{$\infty$} (fid);
	\path (ffid) -- node[below]{$\infty$} (cfid);
	\path (cfid) -- node{$\infty$} (cid);
	\path (fid) -- node[above,pos=.3]{$\infty$} (cid);
	\path (fid) -- node[below,pos=.7]{$\infty$} (pfd);
	\path (cid) -- node{$\infty$} (id);
	\path (pfd) -- node[above,pos=.6]{$\infty$} (id);
	\path (pfd) -- node[below]{$0$} (q-tame);
	\path (id) -- node{$\infty$} (pers);
	\path (q-tame) -- node[above right]{$\infty$} (pers);
	\end{scope}

        \draw [dashed,thick,->] (q-tame) -- node[above,pos=.25]{$\infty$} (cid);
        \draw [dotted,thick,->] (q-tame) -- node[below,pos=.52]{$\infty$} (cfid);
        \draw [dashed,thick,->,transform canvas={xshift=.2em},transform canvas={yshift=-.2em}] (eph) -- node{$0$} (zm);
        \draw [dotted,thick,->] (cfid) -- node{$\infty$} (fid);

\end{tikzpicture}
\caption{Diagram of sets and classes of persistence modules. Solid arrows indicate inclusions, dashed arrows indicate \emph{almost inclusions}, and dotted arrows do not indicate any relationship. Annotations of arrows indicate \emph{enveloping distance} from the source to the target, given in Definitions \ref{def:ed} and \ref{def:red}.}
\label{fig:enveloping-distances}
\end{figure}

\subsection{Sets of persistence modules}
\label{sec:sets}

Next we consider whether the classes defined above are sets or proper classes.
We will use the following notation. 
Let $\overline{\mathbb{R}}:=\mathbb{R}\cup\{\pm\infty\}$ and    $\overline{\mathbb{N}}:=\mathbb{N}\cup \{\infty\}$. 
Given a set $X$, let $\mathcal{P}(X)$ denote its power set.
Let $\mathbb{I}$ be the set of all intervals in $\mathbb{R}$. 
We define a  map, $f: \mathbb{I} \longrightarrow \{1,2,3,4\}$ by
\[ f(I) = 
\begin{cases}  
1, & \inf I \not\in I, \sup I \not\in I \\
2, & \inf I \in I, \sup I \not\in I \\
3, & \inf I \not\in I, \sup I \in I \\
4, & \inf I \in I, \sup I \in I.
\end{cases}
\]

\begin{proposition} \label{prop:cid-set}
  The class $\cid$ is a set.
\end{proposition}

\begin{proof} 
Consider the map 
\[
\cid \longrightarrow \mathcal{P}(\overline{\mathbb{R}}^2\times \{1,2,3,4\} \times \overline{\mathbb{N}}) 
\] 
defined by
\[
\bigoplus_{\alpha \in A} I_{\alpha}  \longmapsto
\bigcup_{\alpha \in A} \big[ \{(\inf I_{\alpha},\sup I_{\alpha})\} \times \{f(I_{\alpha})\} \times \{m(i)\}  \big] 
\] 
where $m(i)$ denotes the multiplicity of the direct summand $I_{\alpha}$. 
This map is an injection, hence $\cid$ is a set.
\end{proof}

\begin{corollary} Therefore the classes $\cfid$, $\fid$, $\ffid$, $\ffids$, and $\zm$ are also sets. 
\end{corollary}

\begin{lemma} \label{lem:pfd}
  Each interval appears only finitely many times in the direct-sum interval-module decomposition of a pointwise finite-dimensional persistence module.
\end{lemma}

\begin{proof}
  For each interval $I$, $\bigoplus_{k=1}^{\infty} I \not\in \pfd$.
\end{proof}

\begin{proposition} \label{prop:pfd-set}
  The class $\pfd$ is a set.
\end{proposition}

\begin{proof}
Let $M \in \pfd$. 
By Theorem~\ref{thm:structure}, $M \isom \bigoplus_{\alpha \in A} I_{\alpha}$ where $I_{\alpha}$ is an interval and $A$ is a set.  
By Lemma~\ref{lem:pfd}, we can define the following map.
\[
\pfd \longrightarrow \mathcal{P}(\overline{\mathbb{R}}^2\times \{1,2,3,4\} \times \mathbb{N}) 
\] 
\[
\bigoplus_{\alpha \in A} I_{\alpha}  \longmapsto
\bigcup_{\alpha \in A} \big[ \{(\inf I_{\alpha},\sup I_{\alpha})\} \times \{f(I_{\alpha})\} \times \{m(i)\}  \big]
\] 
where $m(i)$ denotes the multiplicity of the direct summand $I_{\alpha}$. 
This map is an injection, hence $\pfd$ is a set.
\end{proof}



\begin{proposition}
  The class $\eph$ is not a set.
\end{proposition}

\begin{proof}
  For a cardinal $c$, let $F_c = \bigoplus_{\alpha \in c}[0,0]$. That is, $F_c$ is the $\gf$-vector space generated by $c$.
  For $c \not\isom d$, $F_c \not\isom F_d$. 
  Thus we have an injection from the proper class of cardinals into $\eph$. 
\end{proof}

\begin{corollary} \label{cor:not-set}
  Since $\eph$ is not a set, neither are \id\, \qtame\ and \pers.
\end{corollary}

\subsection{Interval-decomposable persistence modules of arbitrary cardinality}

Motivated by the desire to have a set of persistence modules that contains all of the sets of persistence modules in Section~\ref{sec:sets} and the proofs of Proposition \ref{prop:cid-set} and \ref{prop:pfd-set}, we make the following definition.

\begin{definition} \label{def:kid}
  Given a cardinal $\kappa$, let $\kid$ denote the class of persistence modules isomorphic to $\bigoplus_{\alpha \in A} I_{\alpha}$ where $I_{\alpha}$ is an interval module and the cardinality of $A$ is at most $\kappa$. As a special case, and to avoid confusion with our previously defined notation, let $\rid$ denote the class of interval-decomposable persistence modules with at most the cardinality of $\mathbb{R}$-many summands.
\end{definition}

By definition, $\cid \subset \rid$ and by Lemma~\ref{lem:pfd}, $\pfd \subset \rid$.

\begin{proposition} \label{prop:kid-set}
  For any cardinal $\kappa$, the class $\kid$ is a set.
\end{proposition}

\begin{proof}
  The proof is the same as the proof of Proposition~\ref{prop:cid-set}, replacing $\overline{\mathbb{N}}$ with $\kappa$.
\end{proof}

\section{Topological properties} \label{sec:top}

Since we are interested in studying topological spaces of persistence modules, we will for the most part restrict ourselves to the sets in Figure~\ref{fig:sets}.
We will consider the  basic  topological properties of these sets with the topology induced by the interleaving metric.


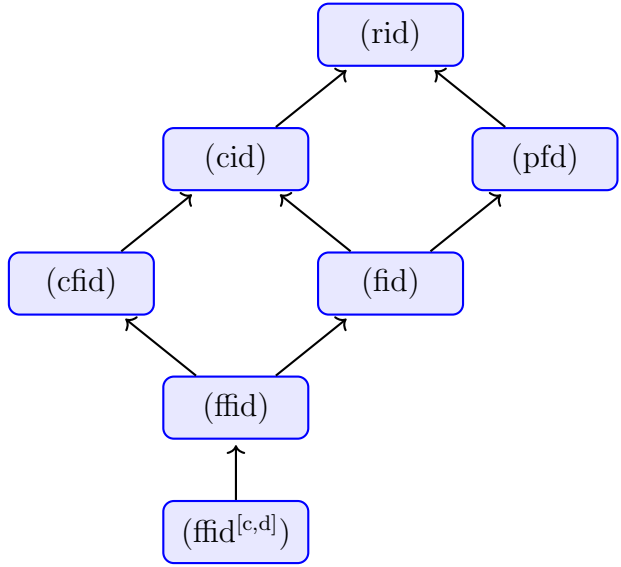
\begin{figure}
\begin{center}
\begin{tikzpicture} [
auto,
decision/.style = { diamond, draw=blue, thick, fill=blue!10,
	text width=4em, text badly centered,
	inner sep=1pt, rounded corners },
block/.style    = { rectangle, draw=blue, thick, 
	fill=blue!9, text width=4em, text centered,
	rounded corners, minimum height=2em },
line/.style     = { draw, thick, ->, shorten >=2pt },
]
\matrix [column sep=1mm, row sep=8mm] {
    & & \node [block] (rid) {$\rid$};\\
    &	\node [block] (cid) {$\cid$}; &  & \node [block] (pfd) {$\pfd$};  &  \\
    \node [block] (cfid) {$\cfid$}; & & \node [block] (fid) {$\fid$}; & &\\
	& \node [block] (ffid) {$\ffid$};  &  & \\	
	& \node [block] (ffids) {$\ffids$};      &     &  \\
};
\begin{scope} [every path/.style=line]
\path (cid) -- (rid);
\path (pfd) -- (rid);
\path (ffids) -- (ffid);
\path (ffid) -- (fid);
\path (ffid) -- (cfid);
\path (cfid) -- (cid);
\path (fid)  -- (cid);
\path (fid)  -- (pfd);
\end{scope}
\end{tikzpicture}
\caption{Sets of metric spaces, each with the topology induced by the interleaving metric.} \label{fig:sets}
\end{center}
\end{figure}

\subsection{Open subsets}
\label{sec:open-subsets}

In this section we consider which of the inclusion maps in Figure~\ref{fig:sets} are inclusions of open subsets.
Recall that in a pseudometric space $X$, a subset $A\subset X$ is said to be open if for all $a \in A$, there exists $\varepsilon>0$ such that $B_\varepsilon(a)\subset A$. \\

\begin{proposition} \label{prop:open}
  Among the inclusion maps in Figure~\ref{fig:sets}, only the inclusions $\ffid \incl \fid$ and $\cfid \incl \cid$ are inclusions of open subsets.
\end{proposition}

\begin{proof}
  Let $M \in \ffid$ and $N \in \fid \setminus \ffid$. Then $N$ is isomorphic to a direct sum of interval modules, at least one of which is unbounded. It follows that $d_I(M,N) = \infty$. Thus $\ffid$ is an open subset of $\fid$. The same argument shows that $\cfid$ is an open subset of $\cid$.
  For each of the following inclusions $\A \subset \B$ we show that for all $M \in \A$ and for all $\eps > 0$, there is an $N \in \B \setminus \A$ such that $d_I(M,N) < \eps$. Therefore $\A$ is not an open subset of $\B$.
  \begin{itemize}
  \item $\ffids \subset \ffid$. Let $N = M \oplus [d,d+2\eps)$.
  \item $\ffid \subset \cfid$. Let $N = M \oplus \bigoplus_{k=1}^{\infty} [0,2\eps)$. 
  \item $\fid \subset \cid$. Let $N = M \oplus \bigoplus_{k=1}^{\infty} [0,2\eps)$.
  \item $\fid \subset \pfd$. Let $N = M \oplus \bigoplus_{k=1}^{\infty} [k,k+2\eps)$.
  \item $\cid \subset \rid$. Let $N = M \oplus \bigoplus_{\mathbb{R}}[0,2\eps)$.
  \item $\pfd \subset \rid$. Let $N = M \oplus \bigoplus_{k=1}^{\infty}[0,2\eps)$. 
  \end{itemize}
\end{proof}

\begin{remark}
  While $\ffids$ is not an open subset of $\ffid$, if we restrict $\ffid$ to direct sums of interval modules whose intervals are contained in an open interval $(c,d)$, then we obtain an open subset of $\ffid$.
\end{remark}



\subsection{Separation}

\begin{proposition} 
Any set of ephemeral persistence modules with the interleaving distance has the indiscrete topology.
\end{proposition}

\begin{proof}
  Let $S$ be a set of ephemeral persistence modules.
  By Lemma~\ref{lem:eph-almost-inclusion}, each $M \in \eph$ has $d_I(M,0)=0$. So for $M,N \in S$, by the triangle inequality, $d_I(M,N)=0$.
  Thus for all $M \in S$ and for all $\eps>0$, $B_{\eps}(M) \supseteq S$.
\end{proof}

\begin{lemma} \label{lem:0-ball}
  Let $M$ be a persistence module let  $r \in \mathbb{R}$. 
  Then $d_I(M,M \oplus [r,r]) = 0$.
\end{lemma}


A topological is said to be a \emph{$T_0$-space} (or a \emph{Kolmogorov space}), if for any pair of distinct elements in the space there exists at least one open set which contains one of them but not the other.

\begin{proposition} 
  Let $c < d$. Then \ffids\ is not a $T_0$-space. 
\end{proposition}
\begin{proof} 
Apply Lemma~\ref{lem:0-ball} to $M=[a,b)$ where $c \leq a < b \leq d$, and $r = \frac{c+d}{2}$.
Then $M' = M \oplus [r,r] \in \ffids$
and there does not exist an open neighborhood $U$ of $M$ that does not contain $M'$ and vice versa. 
\end{proof}

Since \ffids\ is a subspace of any the other spaces in Figure~\ref{fig:sets}, we obtain the following.

\begin{corollary} \label{cor:t0}
  None of the spaces in  Figure~\ref{fig:sets} are $T_0$.
\end{corollary}


\subsection{Compactness}
\label{sec:compactness}

Let $X$ be an extended pseudometric space. Then a subset  $S\subset X$ is \emph{totally bounded} if and only if for each $\eps>0$, there exists a finite subset $F=\{x_1,x_2,\ldots,x_n \} \subset X$  such that $S\subset \cup_{i=1}^n B_\eps(x_i)$. 
Such a union is called a \emph{finite $\eps$-cover}.

\begin{lemma}
  The space $\ffids$ is not totally bounded.
\end{lemma}

\begin{proof}
  Let $\eps < \frac{d-c}{2}$. 
  For $n\geq 0$ consider $M_n = \bigoplus_{k=1}^n [c,d)$. 
  Then for $m\neq n$, $d_I(M_m,M_n) = \frac{d-c}{2}$.
  Therefore $\ffids$ does not have a finite $\eps$-cover.
\end{proof}

An \emph{open cover} of a topological space $X$ is a collection of open sets $\mathcal{O}=\{O_i\}_{i \in I}$ of $X$ such that $\cup_{i\in I} O_i = X$.
A topological spaces is \emph{compact} if every open cover has a finite subcover.
We say that a topological space is \emph{locally compact} if each point has a compact neighborhood, where by a \emph{neighborhood} of a point $p \in X$ we mean a subset $V \subset X$ such that there exists an open set $p \in U \subset V$.

\begin{proposition} 
  Any of element in $\ffids$  does not have a compact neighborhood.
\end{proposition}

\begin{proof} 
Let $M \isom \bigoplus_{j=1}^q I_j$ with $I_j \subset [c,d]$. 
Suppose that $M$ has a compact neighborhood, $K$. 
Then there exists a real number $\eps > 0$ such that $M \in B_\eps(M) \subset K$. 

Choose $\delta>0$ such that $\delta < \eps$, $\delta < d-c$ and $\delta < \frac{1}{4} \min_{j} \diam I_j$.
Choose an interval $I$ of diameter $\delta$ contained in $[c,d]$.
Consider for $n \in \mathbb{N}$, the persistence modules $M_n= M \oplus \bigoplus_{k=1}^n I$.
Then for each $n$, $d_I(M,M_n) \leq \frac{\delta}{2}$ so that the set $\{M_n\}_{n\in \mathbb{N}}$ is contained in $B_\eps(M)$, and hence in $K$. 

Let $M_0 = M$.
Then by the algebraic stability theorem~\cite{ccsggo:interleaving}, $d_I(M_p,M_q) \geq \frac{\delta}{2}$ for all $p>q \geq 0$.
Now consider the open cover $\{B_{\frac{\delta}{6}}(N) \st N \in K\}$ of $K$. 
It does not have a finite subcover, since there does not exist a persistence module $N$ such that $B_{\frac{\delta}{6}}(N)$ contains $M_n$ and $M_m$ for $m \neq n$.
\end{proof}

\begin{corollary} \label{cor:locally-compact}
  All of the spaces in Figure~\ref{fig:sets} are not locally compact. 
\end{corollary}

	

An open covering $\mathcal{O}=\{O_i\}_{i \in I}$ of $X$ is \emph{locally finite} if every $x\in X$ has a neighborhood which has a nonempty intersection with only  finitely many of the open sets $\{O_i\}$.  Given an open cover $\{O_i\}_{i \in I}$ of $X$, another open cover $\mathcal{V}=\{V_j\}_{j\in J}$ is called a \emph{refinement} of $\mathcal{O}$ if for each $V$ in $\mathcal{V}$, there exists $O \in \mathcal{O}$ such that $V \subset O$. A topological space $X$ is said to be a \emph{paracompact} if every open covering admits a locally finite refinement.

\begin{lemma} \label{lem:paracompact}
  An extended pseudometric space is paracompact.
\end{lemma}

\begin{proof}
Let $(X,d)$ be an extended pseudometric space. Let
 $Y = X/{\sim}$ be the quotient space where the equivalence relation $\sim$ is defined on $X$ by $x\sim y \ \Leftrightarrow \ d(x,y)=0$. 
So $(Y,\rho)$ is an extended metric space where $\rho([x],[y])=d(x,y)$.
Let $\pi: X \to Y$ denote the quotient map.
Since $\pi$ maps the open ball $B_r(x)$ to the open ball $B_r([x])$ for all $x \in X$ and all $r>0$, it is an open map.

Now the equivalence relation on $Y$ given by $x \sim y \Leftrightarrow d(x,y)< \infty$ partitions $Y$ into a disjoint union of metric spaces, $Y = \coprod  Y_{\alpha}$.
Given an open cover $\mathcal{U}$ of $Y$, each open set in $\mathcal{U}$ is a disjoint union of open sets, each of which is in one of the $Y_{\alpha}$.
This gives a refinement  of $\mathcal{U}$ that is a disjoint union of open covers of each of the $Y_{\alpha}$.
Each of these metric spaces is paracompact~\cite[Theorem 41.4]{munkres:topology}.
Taking the disjoint union of the resulting locally finite refinements gives the desired locally finite refinement of $Y$.

Let $\mathcal{U}=\{U_i\}_{i \in I}$ be an open cover for $X$. Since $\pi$ is an open map $\{\pi(U_i)\}_{i\in I}$ forms an open cover for $Y$ and since $Y$ is paracompact there is a locally finite refinement $\mathcal{V}=\{V_j\}_{j\in J}$ for $\{\pi(U_i)\}_{i\in I}$. Then the open cover $\pi^{-1}(\mathcal{V})=\{\pi^{-1}(V_j)\}_{j\in J}$ is a locally finite refinement for $\mathcal{U}=\{U_i\}_{i \in I}$. Hence $X$ is paracompact.

\end{proof}

\subsection{Path Connectedness}



\begin{lemma} \label{lem:infinite}
  Let $S$ be an extended pseudometric space. Let $a,b \in S$ with $d(a,b)= \infty$. Then there does not exist a path in $S$ from $a$ to $b$.
\end{lemma}

\begin{proof}
Suppose there is a path $\gamma$ from 
$a$ to $b$ in $S$.
Then $\gamma$ has a compact image.
Therefore the cover $\{B_{1}(x) \st x \in \gamma\}$ should have a finite subcover, which by the triangle inequality contradicts 
$d(a,b) = \infty$.
\end{proof}


\begin{corollary}
  The spaces of persistence modules $\fid$, $\cid$, $\pfd$, $\rid$, and $\cfid$ are not path connected.
\end{corollary}

\begin{proof}
  The first four of these sets contain both $0$ and $[0,\infty)$ and $d_I(0,[0,\infty))= \infty$. The set $\cfid$ contains $0$ and
$\bigoplus_{k=1}^{\infty} [0,k)$ and 
$d_I(\bigoplus_{k=1}^{\infty} [0,k),0) = \infty$.
\end{proof}

\begin{lemma}
  Let $I$ be a finite interval. 
  There exists a path in $\ffids$ from $I$ to the zero module.
\end{lemma}

\begin{proof}
  Let $c = \inf I$ and $d = \sup I$. 
  Let $M^{(0)} = I$ and $M^{(1)} = \Zero$. 
  For $0 < t < 1$, let $M^{(t)} = [c + t\frac{d-c}{2}, d - t\frac{d-c}{2})$.
  Then for $0 \leq s \leq t \leq 1$, $d_I(M^{(s)},M^{(t)}) = (t-s)\frac{d-c}{2}$.
  Thus $\gamma(t) = M^{(t)}$ is a (continuous) path from $I$ to $0$.
\end{proof}

With a similar argument we will show the following.

\begin{proposition} \label{prop:path-component-fid}
  The path component of the zero module in $\fid$ is $\ffid$.
\end{proposition}

\begin{proof}
  By Lemma~\ref{lem:infinite}, the path component of $\Zero$ in $\fid$ is contained in $\ffid$. It remains to show that any $M \in \ffid$ is path connected to $\Zero$.

Let $M \isom \bigoplus_{k=1}^N I_k$, where $I_k$ is a finite interval. For $1 \leq k \leq N$, let $c_k = \inf I_k$ and $d_k = \sup I_k$. 
Let $M^{(0)} = M$ and $M^{(1)}=0$.
For $0<t<1$, let $M^{(t)} = \bigoplus_{k=1}^N [c_k + t \frac{d_k-c_k}{2}, d - t \frac{d_k-c_k}{2})$.
Then for $0 \leq s \leq t \leq 1$, $d_I(M^{(s)},M^{(t)}) \leq (t-s)\max_{1 \leq k \leq N} \frac{d_k-c_k}{2}$.
So $M^{(t)}$ is a continuous path from $M$ to $\Zero$.
\end{proof}

\begin{remark}
  It is not the case that the path component of the zero module in $\cid$ is $\cfid$, since $\cfid$ is not path connected.
Since infinite intervals have infinite distance from the zero module, the path component of the zero module in $\cid$ is the same as the path component of the zero module in $\cfid$.
\end{remark}

\begin{proposition} \label{prop:path-component-cfid}
  The path component of $0$ in $\cfid$, $\pfd$, and $\rid$ consists of modules $\bigoplus_{\alpha \in A} I_{\alpha}$, where $\sup_{\alpha \in A} \length(I_{\alpha}) < \infty$.
\end{proposition}

\begin{proof}
Let $M = \bigoplus_{\alpha \in A} I_{\alpha}$. 
If $\sup_{\alpha \in A} \length(I_{\alpha}) = \infty$ then $d_I(0,M) = \infty$ and $M$ is not in the path component of $0$. 
If $\sup_{\alpha \in A} \length(I_{\alpha}) < \infty$ then the proof of Proposition~\ref{prop:path-component-fid} (replacing $\max$ with $\sup$) shows that $M$ is in the path component of $0$.
\end{proof}


The paths in the previous proposition may be used to show that the following spaces are nullhomotopic.

\begin{proposition} \label{prop:contractible}
   The spaces $\ffids$ and $\ffid$ and the path component of $\Zero$ of $\cfid$, $\pfd$ and $\rid$ are contractible to the zero module.
\end{proposition}

\begin{proof}
  Let $S$ denote either $\ffids$, $\ffid$ or the path component of $\Zero$ in $\cfid$, $\pfd$, or $\rid$.
Assume $M \isom \bigoplus_{k \in A} I_k$, where $A$ is countable.
Let $c_k = \inf I_k$, $d_k = \sup I_k$ and let $h_k = \frac{d_k-c_k}{2}$.
Let $M^{(0)} = M$, $M^{(1)} = 0$ and for $0<t<1$, $M^{(t)} = \bigoplus_{k \in A} [c_k+th_k,d-th_k)$.

We will use these paths to construct a homotopy from the identity map on $S$ to the constant map to the zero module. 
Define $H: S \times [0,1] \to S$ by $(M,t) \mapsto M^{(t)}$.
Let $H_t = H(-,t)$.
Then $H_0 = \Id_S$ and $H_1 = 0$.
It remains to show that $H$ is continuous.
Let $(M,t) \in S\times [0,1]$.
Given $\eps>0$, choose $\delta = \frac{\eps}{1+d_I(M,0)}$.
Let $d$ denote the product metric on $S\times [0,1]$.
Whenever $(N,s) \in S\times [0,1]$ satisfies $d((M,t),(N,s))< \delta$, $d_I(M,N)< \delta$ and $\abs{t-s}<\delta$.
Furthermore
\begin{multline*}
d_I(M^{(t)},N^{(s)}) \leq d_I(M^{(t)},M^{(s)}) + d_I(M^{(s)},N^{(s)}) \\
\leq \abs{t-s} d_I(M,0) + s d_I(M,N)
\leq \delta d_I(M,0) + \delta = \eps
\end{multline*}
which completes the proof. 
\end{proof}

\subsection{Separability}

A topological space is said to be \emph{separable} if it has a countable dense subset.

\begin{theorem} \label{thm:separable}
  The spaces $\fid$, $\ffid$ and $\ffids$ are separable. 
\end{theorem}

\begin{proof} 
First we will show that $\ffid$ is separable.
Let 
\begin{equation} \label{eq:Dn}
  D_n =\big \{\bigoplus_{i=1}^n (p_i,q_i) \in \fid \ | \  p_i, q_i \in \mathbb{Q}, \ p_i < q_i \big\}
\end{equation}
and then consider \[ D=\bigcup_{i=1}^\infty D_n.\] 
Then $D$ is countable and $D$ is dense in $\ffid$ since every open ball of every persistence module in  $\ffid$ contains an element of $D$.

This proof also works for $\fid$ if we allow the intervals in \eqref{eq:Dn} to be infinite, and it works for $\ffids$ if we  restrict the intervals in \eqref{eq:Dn} to be subintervals of $[c,d]$.
\end{proof}

\begin{theorem} \label{thm:not-separable}
The spaces $\cfid$, $\cid$, $\pfd$, and $\rid$ are not separable. 
The same is true for the subspace of $\cid$ with finite distance to $0$ (which equals the subspace of $\cfid$ with finite distance to $0$), and for $\cid \cap \qtame$ and $\cfid \cap \qtame$. 
\end{theorem}
\begin{proof} 
We assign to each binary sequence, $\alpha=(\alpha_n)_{n\geq 1}$ where $\alpha_n\in \{0,1\}$, a persistence module. See Figure~\ref{fig:binary}.
For $n \geq 1$, define 	
\begin{equation*}
I_{n}^{(\alpha)} = \begin{cases}
[2n-1,2n+1),  & \alpha_n=0 \\
[2(n-1),2n+2),  &  \alpha_n=1
\end{cases}
\end{equation*} and let 
$M_\alpha=\bigoplus_{n=1}^{\infty} I_n^{(\alpha)}$.  
Then $M_{\alpha}$ is a persistence module in $\cfid$, $\cid$, $\pfd$, $\rid$, and $\qtame$
and $d_I(M_{\alpha},0) \leq 2$.

\begin{figure}
\begin{center}
\begin{tikzpicture}[point/.style={fill,circle,inner sep=1.7pt,blue}]
\draw[->] (0,0) -- (7,0)node[right]{$x$};
\draw[->] (0,0) -- (0,7)node[above]{$y$};
\foreach \x [evaluate=\x as \lab using int(2*\x)] in {1,2,3,4,5,6}{
\draw[thin](\x,-0.07)--node[below=0.3mm]{$\lab$}++(0,0.2); 
\draw[thin](-0.07,\x)--node[left=0.3mm]{$\lab$}++(0.2,0);  
}
\draw[->,thick,red](0,0)--node[above=3mm,right]{}(7,7); 
\node[point] at (0,2){}; 
\node[point] at (1,3){};
\node[point] at (2,4){};
\node[point] at (3,5){};
\node[point] at (4,6){};
		
\node[point] at (0.5,1.5){};
\node[point] at (1.5,2.5){};
\node[point] at (2.5,3.5){};
\node[point] at (3.5,4.5){};
\node[point] at (4.5,5.5){};
		
\draw[-,thick,blue](4.5,6.5)--node[above right=3mm]{$1$}(5,7); 
\draw[-,thick,blue](5,6)--node[above right=3mm]{$0$}(5.5,6.5); 
\end{tikzpicture}

\caption{Persistence modules corresponding to binary sequences, which are used in the proof of Theorem~\ref{thm:not-separable}.} \label{fig:binary}
\end{center}
\end{figure}
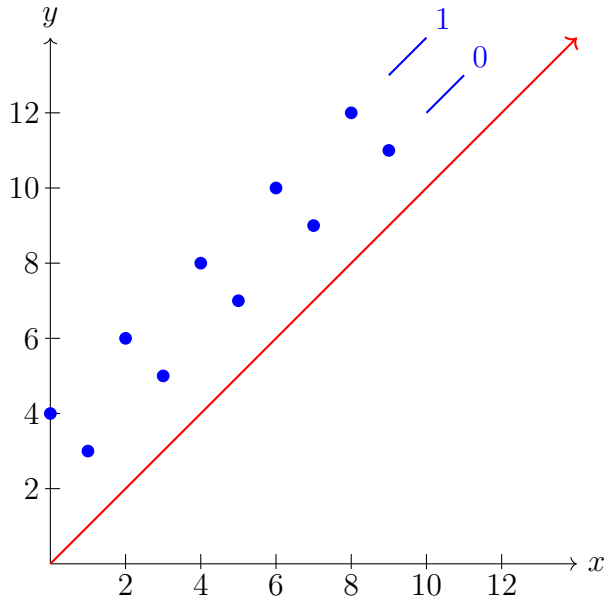

The set $\{M_\alpha \ | \ \alpha  \ \text{is \ a \ binary \ sequence} \}$ is uncountable and for all pairs of binary sequences $\alpha\neq \beta$, we have  $d_I(M_\alpha,M_\beta) = 1$.  Then any dense subset of $\cfid$, $\cid$, $\pfd$, or $\rid$, contains a point in an open ball centered at each $M_\alpha$ of radius $\frac{1}{2}$ and thus cannot be countable.
The same is true 
for the subspace of $\cid$ with finite distance to $\Zero$, and
for $\cid \cap \qtame$ and $\cfid \cap \qtame$. 
\end{proof}

\subsection{Countability}

A topological space is said to be a \emph{first countable} if it has a countable basis at each of its points.

\begin{lemma} \label{lem:first-countable}
  An extended pseudometric space is first countable.
\end{lemma}

\begin{proof}
  Let $x$ be a point in the space. 
  Then the countable collection of open balls $\{B_{\frac{1}{n}}(x) \ | \ n\in \mathbb{N}\}$ is the desired local base at $x$.
\end{proof}

A space $X$ is \emph{compactly generated} if 
a set $A \subset X$ is open if each $A \cap C$ is open in $C$ for each compact subspace $C \subset X$.
Equivalently, a set $B \subset X$ is closed if each $B \cap C$ is closed in $C$ for each compact subspace $C \subset X$.
The following is well known. 

\begin{lemma} \label{lem:compactly-generated}
   If a space is first countable then it is compactly generated.
\end{lemma}

\begin{proof}
  For $B \subset X$, assume that $B \cap C$ is closed in $C$ for each compact subspace $C \subset X$. Let $x$ be a \emph{limit point} of $B$. That is, every neighborhood of $x$ contains point of $B$ other than $x$.
Since $X$ is first countable, there is a sequence of points $(x_i)$ converging to $x$.
Now $(x_i) \cup \{x\}$ is compact, so by assumption $B \cap ((x_i) \cup \{x\})$ is closed in $(x_i) \cup \{x\}$. 
Since $(x_i) \subset B$ it follows that $x \in B$.
Therefore $B$ is closed.
\end{proof}

A topological space is said to be \emph{second countable} if it has a countable basis.
A topological space $X$ is said to be \emph{Lindel\"of} if every open cover of $X$ admits a countable subcover. 

\begin{lemma} \label{lem:tfae}
   For an extended pseudometric space the following properties are equivalent:
   \begin{enumerate}
   \item second countable; 
   \item separable; and
   \item Lindel\"of.
   \end{enumerate}
\end{lemma}

\begin{proof}
Let $X$ be an extended pseudometric space.

$(1) \To (2)$: 
Assume that $X$ has a countable basis $\{B_i\}$. For each $i$, choose $x_i \in B_i$. Then for each $x \in X$ and $r>0$, there exists $i$ such that $B_i \subset B_r(x)$. So $\{x_i\}$ is a countable dense subset of $X$.

$(2) \To (3)$:
Assume that $X$ has a countable dense subset $\{x_i\}$.
Let $\mathcal{U}$ be an open cover of $X$. 
For each $i$, choose $U_i \in \mathcal{U}$ with $x_i \in U_i$. 
Since $U_i$ is open, $U_i \supset B_{r_i}(x_i)$ for some $r_i>0$. 
Since $\{x_i\}$ is dense, $\{U_i\}$ is a countable subcover.

$(3) \To (1)$:  
Assume that $X$ has the Lindel\"of property.
For each $n \geq 1$, 
let $\mathcal{U}_n$ be a countable subcover of the open cover
$\{B_{\frac{1}{n}} (x) \ | \ x\in X\}$.
Then $\mathcal{U} := \cup_n \mathcal{U}_n$
is a countable basis for $X$.
%
\end{proof}

\subsection{Completeness}

An extended pseudometric space is said to be \emph{complete} if every Cauchy sequence converges (see the end of Section~\ref{sec:pseudometric}).

\begin{theorem} \label{thm:not-complete}
The spaces $\pfd$, $\fid$, $\ffid$ and $\ffids$ are not complete. 
\end{theorem}

\begin{proof}  
For $n \geq 0$, let $M_n = \bigoplus_{k=0}^n \left[-\frac{1}{2^k},\frac{1}{2^k}\right)$.
Then the sequence $(M_n) \subset \ffid \subset \fid \subset \pfd$, and $(M_n) \to M = \bigoplus_{k=0}^{\infty} \left[-\frac{1}{2^k},\frac{1}{2^k}\right)$, which is not in $\pfd$.

We claim that there is no $N \in \pfd$ such that $d_I(M,N)=0$.
Assume $N \in \pfd$.
Then $\rank N(0) = R < \infty$.
Thus for all $\eps > 0$, $\rank N(-\eps \leq \eps) \leq R$.
Now for all $\eps>0$, $M$ and $N$ are $\eps$-interleaved,
and thus $\rank M(-2\eps \leq 2\eps) \leq \rank N(-\eps \leq \eps) \leq R$, which is a contradiction.

If we adjust $M_n$ to lie in $[c,d]$, then the same argument shows that $\ffids$ is not complete.
\end{proof}

\begin{theorem} \label{thm:complete} 
  In the class of persistence modules and the class of q-tame persistence modules, every Cauchy sequence has a limit.
Furthermore, the space $\cid \cap \qtame$ is complete,
and so is $\cfid \cap \qtame$.
\end{theorem}

\begin{proof}
Let $(M'_n)_{n\geq 1}$ be a Cauchy sequence of persistence modules.
For each $k \geq 0$, choose a natural number $n_k$ so that  $d_I(M'_m,M'_n)< \frac{1}{2^{k}}$ for all $m,n \geq n_k$. 
Let $M_k$ denote $M'_{n_k}$.
Thus $(M_k)$ is a subsequence of $(M'_n)$ so that for all $k \geq 0$, $M_k$ and $M_{k+1}$ are $\frac{1}{2^{k}}$-interleaved. By the definition of  interleaving, there exist natural transformations $\varphi_{k} : M_k  \Rightarrow M_{k+1}  T_{\frac{1}{2^{k}}} $ and $\psi_{k} : M_{k+1}  \Rightarrow M_k  T_{\frac{1}{2^{k}}}$ such that the triangles corresponding to \eqref{cd:triangle} commute.

Now we define shifted versions of $\varphi$ and $\psi$.
For $k \geq 0$,
let $\alpha^k = \varphi_k T_{-\frac{1}{2^{k-1}}}: M_{k-1} T_{-\frac{1}{2^{k-1}}} \To M_k T_{-\frac{1}{2^k}}$, and
$\beta^k = \psi_k T_{\frac{1}{2^{k}}}: M_{k} T_{\frac{1}{2^{k}}} \To M_{k-1} T_{\frac{1}{2^{k-1}}}$.
Let $a \in \mathbb{R}$.
For every $k \geq 1$, $\alpha^k_a: M_{k-1}(a-\frac{1}{2^{k-1}}) \to M_k(a-\frac{1}{2^k})$
and $\beta^k_a: M_k(a+\frac{1}{2^k}) \to M_{k-1}(a+\frac{1}{2^{k-1}})$.
Thus we have a direct system of vector spaces  
\begin{equation} \label{eq:direct}
  M_0(a - 1) \xto{\alpha^1_a} 
  M_1(a-\frac{1}{2}) \xto{\alpha^2_a} 
  M_2(a-\frac{1}{4}) \xto{\alpha^3_a}
  M_3(a-\frac{1}{8}) \xto{\alpha^4_a} \cdots
\end{equation}
and an inverse system of vector spaces 
\begin{equation} \label{eq:inverse}
  \cdots
  \xto{\beta^4_a} M_3(a+\frac{1}{8})
  \xto{\beta^3_a} M_2(a+\frac{1}{4})
  \xto{\beta^2_a} M_1(a+\frac{1}{2})
  \xto{\beta^1_a} M_0(a+1)
\end{equation}
given in Figure~\ref{fig:Cauchy}.
Note that it follows from the definition of interleaving that each of the trapezoids in Figure~\ref{fig:Cauchy} commute.
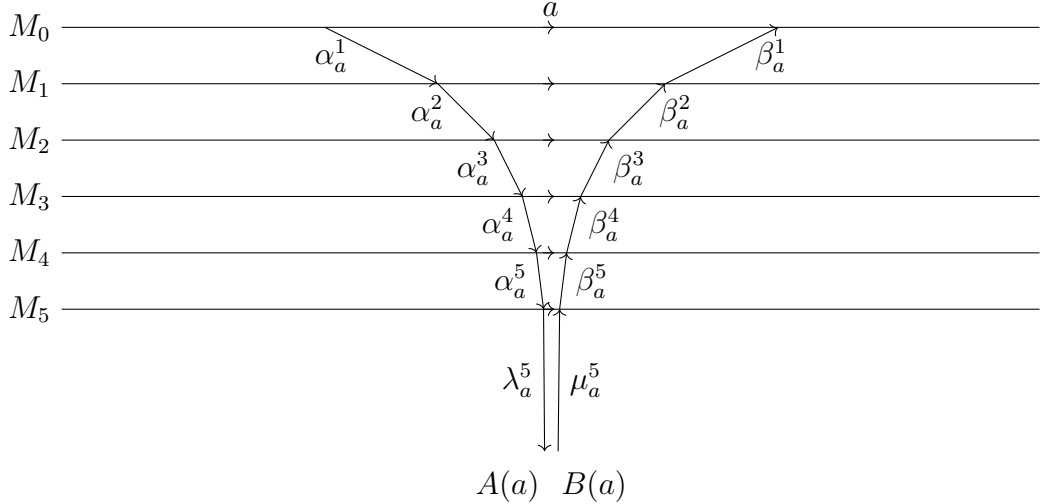
\begin{figure}[h!]
	\begin{tikzpicture}
	\draw (8,0) -- (21,0);
	\draw (8,-0.75) -- (21,-0.75);
	\draw (8,-1.5) -- (21,-1.5);
	\draw (8,-2.25) -- (21,-2.25);
	\draw (8,-3) -- (21,-3);
	\draw (8,-3.75) -- (21,-3.75);
	
	\node[left] at (8,0) {$M_0$};
	\node[left] at (8,-0.75) {$M_1$};
	\node[left] at (8,-1.5) {$M_2$};
	\node[left] at (8,-2.25) {$M_3$};
	\node[left] at (8,-3) {$M_4$};
	\node[left] at (8,-3.75) {$M_5$};
	\node[above] at (14.5,0) {$a$};
	\node[below left] at (14.5,-5.75) {$A(a)$};
	\node[below right] at (14.5,-5.75) {$B(a)$};
	
	\draw[->] (11.5,0) -- node[left=3mm]{$\alpha_a^1$} (13,-0.75);
	\draw[->] (13,-0.75) -- node[left=1.5mm]{$\alpha_a^2$} (13.75,-1.5);
	\draw[->] (13.75,-1.5) -- node[left=1mm]{$\alpha_a^3$} (14.125,-2.25);
	\draw[->] (14.125,-2.25) -- node[left=0.5mm]{$\alpha_a^4$} (14.3125,-3);
	\draw[->] (14.3125,-3) -- node[left=0.25mm]{$\alpha_a^5$} (14.40625,-3.75);
	\draw[->] (14.40625,-3.75) -- node[left]{$\lambda_a^5$} (14.42,-5.64);
	
	\draw[->] (14.6,-5.64) -- node[right]{$\mu_a^5$} (14.62,-3.75);
	\draw[->] (14.62,-3.75) -- node[right=0.25mm]{$\beta_a^5$} (14.71325,-3);
	\draw[->] (14.71325,-3) -- node[right=0.5mm]{$\beta_a^4$} (14.90075,-2.25);
	\draw[->] (14.90075,-2.25) -- node[right=1mm]{$\beta_a^3$} (15.27575,-1.5);
	\draw[->] (15.27575,-1.5) -- node[right=1.5mm]{$\beta_a^2$} (16.02575,-0.75);
	\draw[->] (16.02575,-0.75) -- node[right=3mm]{$\beta_a^1$} (17.52575,0);
	
	\draw[->]  (14.4,0) -- (14.55,-0);
	\draw[->]  (14.4,-0.75) -- (14.55,-0.75);
	\draw[->]  (14.4,-1.5) -- (14.55,-1.5);
	\draw[->]  (14.4,-2.25) -- (14.55,-2.25);
	\draw[->]  (14.4,-3) -- (14.55,-3);
	\draw[->]  (14.4,-3.75) -- (14.55,-3.75);
	
	\end{tikzpicture}
\caption{A direct system of vector spaces and an inverse system of vector spaces in a Cauchy sequence of persistence modules.}
\label{fig:Cauchy}
\end{figure}
Let $A(a)$ be the colimit (i.e. direct limit) of $\eqref{eq:direct}$, and
let $B(a)$ be the limit (i.e. inverse limit)  of $\eqref{eq:inverse}$.
For each $k \geq 0$, we have maps 
$\lambda^k_a: M_k(a-\frac{1}{2^k}) \to A(a)$ and
$\mu^k_a: B(a) \to M_k(a+\frac{1}{2^k})$.
By the universal properties of the colimit and the limit, we have a map $\theta_a: A(a) \to B(a)$, and
\begin{equation}
  \label{eq:lambda-theta-mu}
  \mu^k_a \theta_a \lambda^k_a = M_k(a-\textstyle\frac{1}{2^k} \leq a+\textstyle\frac{1}{2^k}).
\end{equation}
Let $M(a)$ denote the image of $\theta_a$.
Thus, $\theta_a$ factors as follows.
\begin{equation} \label{cd:theta-a}
  \begin{tikzcd}[row sep=small]
    A(a) \arrow[rr, "\theta_a"] \ar[dr,"\rho_a"'] &  & B(a)\\
    & M(a) \arrow[ur,"\iota_a"']
  \end{tikzcd}
\end{equation}

Now observe that all of these constructions are functorial.
Thus, we have persistence modules $A$, $B$ and $M$.
We also have natural transformations $\lambda^k: M_k T_{-\frac{1}{2^k}} \To A$  
and
$\mu^k: B \To M_k T_{\frac{1}{2^k}}$. 
In addition we have the following commutative diagram of natural transformations.
\begin{equation*} 
  \begin{tikzcd}[row sep=small]
    A \arrow[rr, "\theta"] \ar[dr,"\rho"'] &  & B\\
    & M \arrow[ur,"\iota"']
  \end{tikzcd}
\end{equation*}
These fit into the commutative diagram in Figure~\ref{fig:Cauchy-with-limit}, where we have corresponding arrows for all $a \in \mathbb{R}$.
\begin{figure}[h!]
	\begin{tikzpicture}
	\draw (8,0) -- (21,0);
	\draw (8,-0.75) -- (21,-0.75);
	\draw (8,-1.5) -- (21,-1.5);
	\draw (8,-2.25) -- (21,-2.25);
	\draw (8,-3) -- (21,-3);
	\draw (8,-3.75) -- (21,-3.75);
	\draw (8,-5.75) -- (21,-5.75);
	
	\node[left] at (8,0) {$M_0$};
	\node[left] at (8,-0.75) {$M_1$};
	\node[left] at (8,-1.5) {$M_2$};
	\node[left] at (8,-2.25) {$M_3$};
	\node[left] at (8,-3) {$M_4$};
	\node[left] at (8,-3.75) {$M_5$};
	\node[left] at (8,-5.75) {$A, B, M, \rad M$};
	\node[below] at (14.5,-5.75) {$a$};
	
	\draw[->] (11.5,0) -- (13,-0.75);
	\draw[->] (13,-0.75) -- (13.75,-1.5);
	\draw[->] (13.75,-1.5) -- (14.125,-2.25);
	\draw[->] (14.125,-2.25) -- (14.3125,-3);
	\draw[->] (14.3125,-3) -- (14.40625,-3.75);
	\draw[->] (14.40625,-3.75) -- (14.42,-5.64);
	
	\draw[->] (14.6,-5.64) -- (14.62,-3.75);
	\draw[->] (14.62,-3.75) -- (14.71325,-3);
	\draw[->] (14.71325,-3) -- (14.90075,-2.25);
	\draw[->] (14.90075,-2.25) -- (15.27575,-1.5);
	\draw[->] (15.27575,-1.5) -- (16.02575,-0.75);
	\draw[->] (16.02575,-0.75) -- (17.52575,0);
	
	\draw[->]  (14.4,0) -- (14.55,-0);
	\draw[->]  (14.4,-0.75) -- (14.55,-0.75);
	\draw[->]  (14.4,-1.5) -- (14.55,-1.5);
	\draw[->]  (14.4,-2.25) -- (14.55,-2.25);
	\draw[->]  (14.4,-3) -- (14.55,-3);
	\draw[->]  (14.4,-3.75) -- (14.55,-3.75);
	\end{tikzpicture}
			
\caption{A particular subsequence of a Cauchy sequence of persistence modules and some persistence modules in the limit.}
\label{fig:Cauchy-with-limit}
\end{figure}
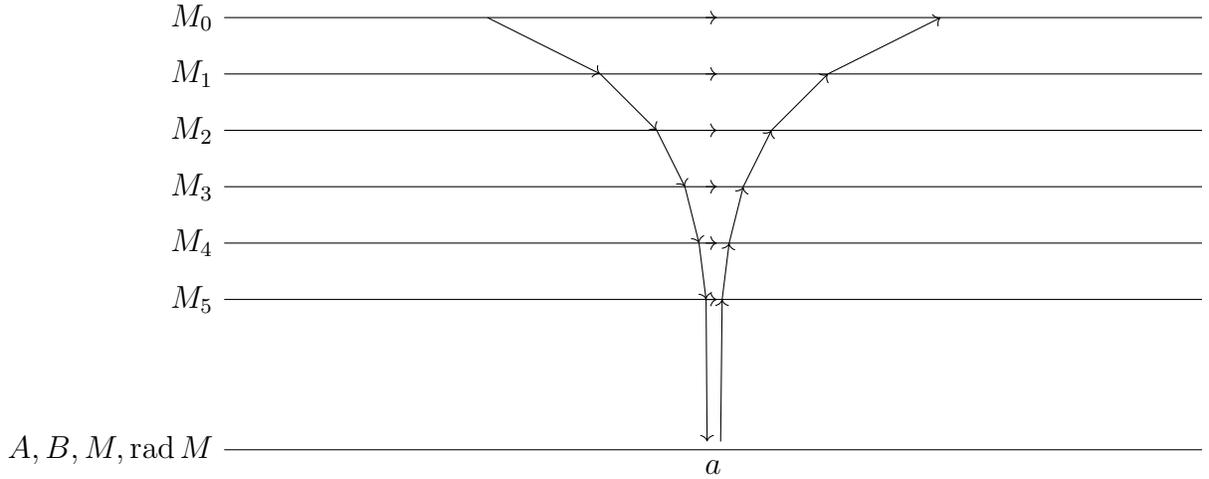

Let $a \in \mathbb{R}$ and $k\geq 1$. 
Define $b = a + \frac{1}{2^{k-1}}$.
Then we have the following bi-infinite sequence.
\begin{multline} \label{eq:nu}
  \cdots 
  \xto{\beta_a^{k+3}} M_{k+2}(a+\frac{1}{2^{k+2}}) 
  \xto{\beta_a^{k+2}} M_{k+1}(a+\frac{1}{2^{k+1}}) 
  \xto{\beta_a^{k+1}} M_{k}(a+\frac{1}{2^{k}}) \\
  \xto{\alpha_b^{k+1}} M_{k+1}(b-\frac{1}{2^{k+1}}) 
  \xto{\alpha_b^{k+2}} M_{k+2}(b-\frac{1}{2^{k+2}}) 
  \xto{\alpha_b^{k+3}} \cdots
\end{multline}
Notice that the left part of this sequence is an initial part of \eqref{eq:inverse}
and the right part of this sequence is a terminal part of \eqref{eq:direct}.
It follows that \eqref{eq:nu} has limit $B(a)$ and colimit $A(b)$, and there is an induced map
$\nu:B(a) \to A(b)$.
We obtain the commutative diagram in Figure~\ref{fig:horseshoe}.

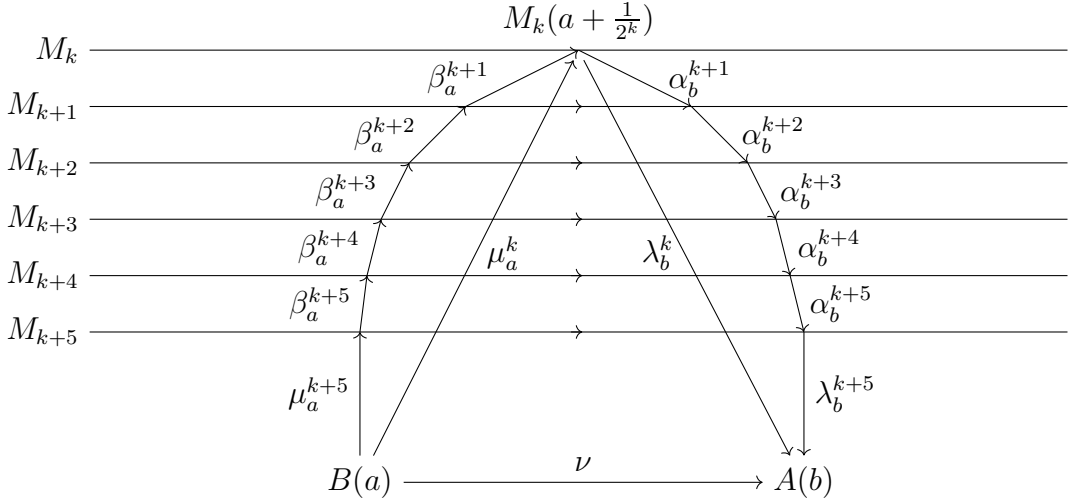
\begin{figure}[h!]
	\begin{tikzpicture}
	\draw (8,0) -- (21,0);
	\draw (8,-0.75) -- (21,-0.75);
	\draw (8,-1.5) -- (21,-1.5);
	\draw (8,-2.25) -- (21,-2.25);
	\draw (8,-3) -- (21,-3);
	\draw (8,-3.75) -- (21,-3.75);
	
	\node[left] at (8,0) {$M_k$};
	\node[left] at (8,-0.75) {$M_{k+1}$};
	\node[left] at (8,-1.5) {$M_{k+2}$};
	\node[left] at (8,-2.25) {$M_{k+3}$};
	\node[left] at (8,-3) {$M_{k+4}$};
	\node[left] at (8,-3.75) {$M_{k+5}$};
	\node[above] at (14.5,0) {$M_k(a+\frac{1}{2^k})$};
	\node (B) at (11.59375,-5.75) {$B(a)$};
	\node (A) at (17.49625,-5.75) {$A(b)$};
	
	\draw[<-] (14.5,0) -- node[left=3mm]{$\beta_a^{k+1}$} (13,-0.75);
	\draw[<-] (13,-0.75) -- node[left=1.5mm]{$\beta_a^{k+2}$} (12.25,-1.5);
	\draw[<-] (12.25,-1.5) -- node[left=1mm]{$\beta_a^{k+3}$} (11.875,-2.25);
	\draw[<-] (11.875,-2.25) -- node[left=0.5mm]{$\beta_a^{k+4}$} (11.6875,-3);
	\draw[<-] (11.6875,-3) -- node[left=0.25mm]{$\beta_a^{k+5}$} (11.59375,-3.75);
	\draw[<-] (11.59375,-3.75) -- node[left]{$\mu_a^{k+5}$} (B);
	\draw[->] (14.5,0) -- node[right=3mm]{$\alpha_b^{k+1}$} (16,-0.75);
	\draw[->] (16,-0.75) -- node[right=1.5mm]{$\alpha_b^{k+2}$} (16.75,-1.5);
	\draw[->] (16.75,-1.5) -- node[right=1mm]{$\alpha_b^{k+3}$} (17.125,-2.25);
	\draw[->] (17.125,-2.25) -- node[right=0.5mm]{$\alpha_b^{k+4}$} (17.3125,-3);
	\draw[->] (17.3125,-3) -- node[right=0.25mm]{$\alpha_b^{k+5}$} (17.49625,-3.75);
	\draw[->] (17.49625,-3.75) -- node[right]{$\lambda_b^{k+5}$} (A); 
	
	\draw[->]  (14.4,-0.75) -- (14.55,-0.75);
	\draw[->]  (14.4,-1.5) -- (14.55,-1.5);
	\draw[->]  (14.4,-2.25) -- (14.55,-2.25);
	\draw[->]  (14.4,-3) -- (14.55,-3);
	\draw[->]  (14.4,-3.75) -- (14.55,-3.75);
	\draw[->] (B) -- node[above]{$\nu$} (A);
        \draw[shorten >=1.5mm,->] (B) -- node[right]{$\mu_a^k$} (14.5,0);
        \draw[shorten >=1.5mm,<-] (A) -- node[left]{$\lambda_b^k$} (14.5,0);
	\end{tikzpicture}
\caption{The bi-infinite sequence in \eqref{eq:nu}, its limit and colimit, and three induced maps.}
\label{fig:horseshoe}
\end{figure}

By the universal properties of limit and colimit, we have the following commutative diagram.
\begin{equation*}
  \begin{tikzcd}
    A(a) \ar[d,"\theta_a"'] \ar[r,"A(a\leq b)"] & A(b) \ar[d,"\theta_b"] \\
    B(a) \ar[r,"B(a\leq b)"'] \ar[ur,"\nu"] & B(b)
  \end{tikzcd}
\end{equation*}
By the commutativity of the bottom right part of this diagram, we have that $\im B(a \leq b) \subset \im \theta_b$.
So we have the following commutative diagram.
\begin{equation*}
  \begin{tikzcd}[row sep=scriptsize]
    M(a) = \im \theta_a \ar[r,hookrightarrow,"\iota_a"] \ar[drr,"M(a\leq b)"'] & B(a) \ar[r,"\nu"] \ar[dr,"B(a\leq b)"] & A(b) \ar[d,"\rho_b"] \\
    & & \im \theta_b = M(b)
  \end{tikzcd}
\end{equation*}
Thus
\begin{equation} \label{eq:nu2}
  \rho_b \nu \iota_a = B(a \leq b)|_{\im \theta_a} = M(a\leq b).
\end{equation}

Now consider the following natural transformations.
\begin{equation} \label{eq:interleaving1}
  (\rho \lambda^k) T_{\frac{1}{2^k}} : M_k \To M T_{\frac{1}{2^k}} 
\end{equation}
\begin{equation} \label{eq:interleaving2}
  \mu^k \iota : M \To M_k T_{\frac{1}{2^k}}
\end{equation}
We claim that these natural transformations provide an interleaving (Section~\ref{sec:interleaving}).
That is, 
\begin{equation} \label{eq:interleaving-identities}
((\mu^k \iota) T_{\frac{1}{2^k}}) ((\rho \lambda^k) T_{\frac{1}{2^k}}) = 
M_k \eta_{\frac{1}{2^{k-1}}},
\text{ and }
((\rho \lambda^k) T_{\frac{1}{2^{k-1}}}) (\mu^k \iota) = M \eta_{\frac{1}{2^{k-1}}},
\end{equation}
where $\eta$ is the natural transformation defined in Section~\ref{sec:interleaving}.

A pair of natural transformations are equal if and only if their components are equal.
We remark that for a natural transformations $\alpha$ and $\beta$,
the natural transformation $\alpha T_x$ has components 
$(\alpha T_x)_a = \alpha_{a+x}$,
and the natural transformation $\beta \alpha$ has components $(\beta \alpha)_a = \beta_a \alpha_a$.

Let $a \in \mathbb{R}$. We will verify the identities in  \eqref{eq:interleaving-identities} using the $a$ component. For the left hand side of the first identity, we have
\begin{equation*}
  M(a) \xto{\lambda_{a+\frac{1}{2^k}}^k} A(a+\frac{1}{2^k}) \xto{\rho_{a+\frac{1}{2^k}}} M(a+\frac{1}{2^k}) \xto{\iota_{a+\frac{1}{2^k}}} B(a+\frac{1}{2^k}) \xto{\mu_{a+\frac{1}{2^k}}^k} M(a+\frac{1}{2^{k-1}}) = M(b).
\end{equation*}
Using \eqref{cd:theta-a} the composition of the inner two maps equals $\theta_{a+\frac{1}{2^k}}$.
Then using \eqref{eq:lambda-theta-mu} we see that the entire composition equals $M_k(a \leq a+\frac{1}{2^{k-1}})$, as desired.

For the left hand side of the second identity, we have
\begin{equation*}
  M(a) \xto{\iota_a} B(a) \xto{\mu_a^k} M(a+\frac{1}{2^k}) \xto{\lambda_{a+\frac{1}{2^{k-1}}}^k} A(a+\frac{1}{2^{k-1}}) \xto{\rho_{a+\frac{1}{2^{k-1}}}^k} M(a+\frac{1}{2^{k-1}}) = M(b).
\end{equation*}
Using the commutativity of the induced maps in Figure~\ref{fig:horseshoe}, the composition of the inner two maps equals $\nu$.
Then using \eqref{eq:nu2}, we see that the entire composition equals $M(a \leq b)$, as desired.

Thus \eqref{eq:interleaving1} and \eqref{eq:interleaving2} is a $\frac{1}{2^k}$-interleaving.
Therefore $M$ is a limit of the sequence $(M_k)$ and hence also a limit of the Cauchy sequence $(M'_n)$.
Thus any Cauchy sequence of persistence modules has a limit.

Now assume that each of the $M'_n$ are in $\qtame$.
We will show that $M \in \qtame$.
Let $a<b$, and choose $k\geq 0$ so that $\frac{1}{2^{k-1}} < b-a$.
Then the following diagram commutes.
\begin{equation*}
  \begin{tikzcd}[row sep=scriptsize]
    & M_k(a+\frac{1}{2^k}) \ar[rrr,"M_k(a+\frac{1}{2^k} < b-\frac{1}{2^k})"] & & & M_k(b-\frac{1}{2^k}) \ar[dr,"\rho_b \lambda_b^k"] \\
M(a) \ar[ur,"\mu_a^k \iota_a"] \ar[rrrrr,"M(a<b)"] & & & & & M(b)
  \end{tikzcd}
\end{equation*}
Since $M_k$ is q-tame, the top horizontal arrow has finite rank, and hence so does the bottom horizontal arrow.
Thus $M \in \qtame$. 
Therefore any Cauchy sequence of q-tame persistence modules has a limit.

Now since $M \in \qtame$,
by Theorem~\ref{thm:rad}, $\rad M \in \cid \cap \qtame$.
By Proposition~\ref{prop:rad}, $d_I(M, \rad M) = 0$.
Therefore by the triangle inequality, $\rad M$ is also a limit of the Cauchy sequence.
Thus $\cid \cap \qtame$ is complete.

Finally, assume that in addition, each $M'_n \in \cfid \cap \qtame$.
Since $M$ is $\frac{1}{2^k}$-interleaved with $M_k$, which does not contain any infinite intervals in its direct sum decomposition, neither does $M$.
Therefore $\rad M$ also does not contain any infinite intervals in its direct sum decomposition.
That is, $\rad M \in \cfid \cap \qtame$.
\end{proof}

Now we present a second, more concise proof of the main result in the previous proof.

\begin{proof}
  We may consider the diagram in Figure~\ref{fig:Cauchy} to be a functor $M: (\mathbb{R} \times \mathbb{N},\leq) \to \vect$, where 
$(\mathbb{R} \times \mathbb{N},\leq)$ is the poset generated by the inequalities $(a,k) \leq (b,k)$, where $a \leq b \in \mathbb{R}$ and $k \geq 0$, and $(a,k-1) \leq (a+\frac{1}{2^k},k)$ and $(a,k) \leq (a+\frac{1}{2^k},k-1)$, where $a \in \mathbb{R}$ and $k \geq 1$.

Now extend this poset to $(\mathbb{R} \times \overline{\mathbb{N}} , \leq)$, by adding the generating inequalities $(a,\infty) \leq (b,\infty)$ for all $a\leq b$, and $(a-\frac{1}{2^k},k) \leq (a,\infty)$ and $(a,\infty) \leq (a+\frac{1}{2^k},k)$.
\begin{equation*}
  \begin{tikzcd}[row sep=scriptsize]
    (\mathbb{R} \times \mathbb{N}, \leq) \arrow[r,"M"] \ar[d,hookrightarrow,"i"'] & \vect \\
    (\mathbb{R} \times \overline{\mathbb{N}},\leq) \ar[ur,dashrightarrow]
  \end{tikzcd}
\end{equation*}
We can extend the functor $M$ to $(\mathbb{R} \times \overline{\mathbb{N}},\leq)$ by taking either the left or the right Kan extension. We obtain functors corresponding to the diagram in Figure~\ref{fig:Cauchy-with-limit}, where $A = \Lan_{i}M$ and $B = \Ran_{i}M$. Then there is a canonical map $\theta: A \To B$, and the image of this map gives another extension of $M$. Abusing notation, let $M = \im(\theta): (\mathbb{R} \times \overline{\mathbb{N}},\leq) \to \vect$.

For $k \in \overline{\mathbb{N}}$, let $M_k = M(-,k)$. Then by construction, $M_{\infty}$ is $\frac{1}{2^k}$-interleaved with $M_k$.
Thus, $M_{\infty}$ is a limit of the Cauchy sequence.
\end{proof}

\subsection{Baire spaces}

Let $X$ be a topological space. A subspace $A \subset X$ has empty interior in $X$ if $A$ does not contain an open set in $X$.
The space $X$ is said to be a \emph{Baire space} if for any countable collection of closed sets in $X$ with empty interior in $X$, their union also has empty interior in $X$.

\begin{theorem}[Baire category theorem] \label{thm:baire}
  A complete extended pseudometric space is a Baire space.
\end{theorem}

\begin{proof}
Let $X$ be an extended pseudometric space.
Let $\{A_n\}$ be a countable collection of closed sets in $X$ with empty interior in $X$.
We want to show that $\bigcup A_n$ has empty interior in $X$.
Let $U$ be an open set in $X$.
We will show that $U \not\subset \bigcup A_n$.
We need an $x \in U$ such that for all $n$, $x \not\in A_n$.
By assumption, there is a $x_1 \in U$ with $x_1 \not\in A_1$.
Since $U$ is open and $A_1$ is closed, there is an $r_1 \leq 1$ such that $B_{r_1}(x_1) \subset U$ and $B_{r_1}(x_1) \cap A_1 = \emptyset$.
Let $s_1 = \frac{r_1}{2}$.
Then $\overline{B_{s_1}(x_1)} \subset U$ and $\overline{B_{s_1}(x_1)} \cap A_1 = \emptyset$.
Given $B_{s_n}(x_n)$ with $\overline{B_{s_n}(x_n)} \cap A_n = \emptyset$,
then by assumption, there is a $x_{n+1} \in B_{s_n}(x)$ with $x_{n+1} \not\in A_{n+1}$.
Since $B_{s_n}(x)$ is open and $A_{n+1}$ is closed, there is an $r_{n+1} \leq \frac{1}{n+1}$ with $B_{r_{n+1}}(x_{n+1}) \subset B_{s_n}(y_n)$ and $B_{r_{n+1}}(x_{n+1}) \cap A_{n+1} = \emptyset$.
Let $s_{n+1} = \frac{r_{n+1}}{2}$.
Then $\overline{B_{s_{n+1}}(x_{n+1})} \subset \overline{B_{s_n}(y_n)}$ and $\overline{B_{s_{n+1}}(x_{n+1})} \cap A_{n+1} = \emptyset$.
Since $\overline{B_{s_1}(x_1)} \supset \overline{B_{s_2}(x_2)} \supset \overline{B_{s_3}(x_3)} \supset \cdots$ and $(s_n) \to 0$, $(x_n)$ is a Cauchy sequence in $X$.
Since $X$ complete, there exists a $x \in X$ such that $(x_n) \to x$. 
Since $x_n \in \overline{B_{s_1}(x_1)}$ for all $n$, $x \in \overline{B_{s_1}(x_1)} \subset U$.
Also, for all $n$, the sequence $x_n,x_{n+1},x_{n+2},\ldots$ in $\overline{B_{s_n}(x_n)}$ converges to $x$, so $x \in \overline{B_{s_n}(x_n)}$. Thus $x \not\in A_n$ for all $n$.
\end{proof}

\begin{corollary} \label{cor:baire}
  Hence $\cid \cap \qtame$ and $\cfid \cap \qtame$ are Baire spaces.
\end{corollary}

\subsection{Topological dimension} \label{sec:dim}

Let $X$ be a topological space.
A collection of subsets of $X$ has \emph{order} $m$ if there is a point in $X$ contained in $m$ of the subsets, but no point of $X$ is contained in $m+1$ of the subsets.
The \emph{topological dimension} of $X$ (also called the Lebesgue covering dimension) is the smallest number $m$ such that every open cover of $X$ has a  refinement (see Section~\ref{sec:compactness}) with order $m+1$.

\begin{theorem} \label{thm:embedding-cube}
  Let $N \geq 1$.  There exists an $\eps>0$ such that there is an isometric embedding of the cube $[0,\eps]^N$ with the $L^{\infty}$ distance into $\ffids$.
\end{theorem}

\begin{proof}
  Assume $[c,d] = [0,1]$. The proof for the general case is similar.
Choose $\eps < \frac{1}{100N}$. 
Let $x = (x_1,\ldots,x_N) \in [0,\eps]^N$. 
We will define a map $x \mapsto M = M(x) = \bigoplus_{i=1}^N I_i$, where each interval $I_i = I_i(x_i)$ depends only on $x_i$.
We will choose $I_1,\ldots,I_N$ to be far from each other and far from the zero module but so that $I_i(x_i)$ is close to $I_i(x'_i)$ for any $x_i, x'_i \in [0,\eps]$.

For $1 \leq i \leq N$, let $I_i = \left[ \frac{i}{N}, \frac{i}{N} + \frac{1}{10N} + x_i \right)$.
Then $d_I(I_i(x_i),I_i(x'_i)) = \abs{x_i-x'_i} \leq \frac{1}{100N}$.
Also $d_I(I_i,0) \geq \frac{1}{20N}$.
Since for $i \neq j$, $I_i$ and $I_j$ are disjoint, 
and so we also have that $d_I(I_i,I_j) \geq \frac{1}{20N}$.
Therefore $d_I(M(x),M(x')) = \norm{x-x'}_{\infty}$.
\end{proof}

\begin{corollary} \label{cor:top-dim}
  The topological dimension of all of the topological spaces of persistence modules in Figure~\ref{fig:sets} is infinite.
\end{corollary}

\begin{proof}
  Let $X$ be one of the spaces in Figure~\ref{fig:sets}. 
  Then by the previous theorem, for all $N \geq 1$, $\dim X \geq \dim [0,\eps]^N = N$. Thus $\dim X = \infty$.
\end{proof}

\section{Open questions}
\label{sec:questions}

We end with some unresolved questions.
\begin{itemize}
\item Are \cid\ and \cfid\ complete?
\item Can the results presented here be extended to multiparameter persistence modules and generalized persistence modules?
\end{itemize}

\subsection*{Acknowledgments}

The authors would like to that the anonymous referees for their helpful suggestions. 
In particular, we would like to thank the referee who contributed the proof that the enveloping distance from pointwise-finite dimensional persistence modules to q-tame persistence modules is zero.
We also thank Alex Elchesen for proofreading an earlier draft of the paper.
The first author would like to acknowledge the support of the Army Research Office, Award W911NF1810307, and the Southeast Center for Mathematics and Biology, an NSF-Simons Research Center for Mathematics of Complex Biological Systems, under National Science Foundation Grant No. DMS-1764406 and Simons Foundation Grant No. 594594.



\appendix

\section{The arithmetic of maps and interleavings of interval modules}
\label{sec:interval-arithmetic}

In this appendix, we give some basic results on interval modules, maps of interval modules, interleavings of interval modules, and neighborhoods of interval modules.

\subsection{Some relations between intervals}
\label{sec:interval-relations}

First we define some relations between intervals that will be useful in the following sections and describe some of their properties.

Recall that $I \subset \mathbb{R}$ is an \emph{interval} if $a, c \in I$ and $a \leq b \leq c$ then $b \in I$. It follows that the intersection of two intervals is an interval.

\begin{definition} \label{def:interval-po}
  For $A,B \subset \mathbb{R}$, define the relation $A \leq B$ if 
  \begin{enumerate}
  \item \label{it:int-po1} for all $a \in A$ there is a $b \in B$ such that $a \leq b$, and
  \item \label{it:int-po2} for all $b \in B$ there is an $a \in A$ such that $a \leq b$.
  \end{enumerate}
\end{definition}

\begin{lemma}
  This relation defines a partial order on intervals.
\end{lemma}

\begin{proof}
  Let $A$, $B$, and $C$ be intervals. $A \leq A$ since for all $a \in A$, $a \leq a$. Assume $A \leq B$ and $B \leq A$. Let $a \in A$. Then 
by Definition \ref{def:interval-po} (\ref{it:int-po1}), there is $b \in B$ with $a \leq b$, and 
by Definition \ref{def:interval-po} (\ref{it:int-po2}), there is $b' \in B$ with $b' \leq a$. Since $B$ is an interval $a \in B$. Thus $A \subset B$. Similarly $B \subset A$.

Finally assume $A \leq B$ and $B \leq C$. 
For all $a \in A$ there is a $b \in B$ with $a \leq b$ and $c \in C$ with $b \leq c$. Thus $a \leq c$.
For all $c \in C$ there is a $b \in B$ with $b \leq c$ and $a \in A$ with $a \leq b$. Thus $a \leq c$. 
Therefore $A \leq C$.
\end{proof}

Let us define another relation.

\begin{definition}
  For $A,B \subset \mathbb{R}$, define $A \prec B$ if for all $a \in A$ and $b \in B$, $a \leq b$.
\end{definition}


\begin{lemma} \label{lem:interval-to}
   Let $I$ and $J$ be disjoint, nonempty intervals. Then $J \leq I$ iff $J \prec I$.
\end{lemma}

\begin{proof}
  See Figure~\ref{fig:disjoint}.
  Let $j \in J$. Then either condition implies that there is an $i \in I$ with $j \leq i$. The negation of either condition implies that there is an $i \in I$ with $i<j$. Since $I$ is an interval, this would imply that $j \in I$ which is a contradiction.
\end{proof}

\begin{figure}[h!]
	\begin{tikzpicture}
	\draw (13,0) -- (20,0);
	\draw (8.3,-0.75) -- (12,-0.75);

	\node[left] at (8,0) {$I$};
	\node[left] at (8,-0.75) {$J$};
	\end{tikzpicture}
        \caption{Two disjoint nonempty intervals.}
        \label{fig:disjoint}
\end{figure}
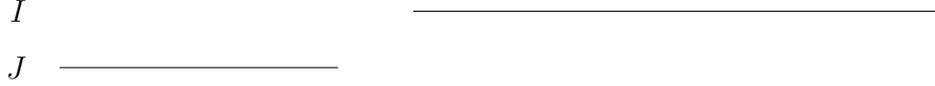

\begin{lemma} \label{lem:all-intervals}
  If $J$ and $I$ are intervals with $J \leq I$ then $J \setminus (I \cap J) = J \setminus I$ is an interval and $I \setminus (I \cap J) = I \setminus J$ is an interval.
\end{lemma}

\begin{proof}
  Let $a,c \in J \setminus (I \cap J)$ and $a \leq b \leq c$.
  Since $J$ is an interval, $b \in J$. 
  Since $c \in J$ there is a $d \in I$ with $c \leq d$. Since $c \not\in I$ and $I$ is an interval, $b \not\in I$. Thus $b \in J \setminus (I \cap J)$.

  Let $a,c \in I \setminus (I \cap J)$ and $a \leq b \leq c$.
  Since $I$ is an interval $b \in I$.
  Since $a \in I$ there is a $x \in J$ with $x \leq a$. Since $a \not\in J$ and $J$ is an interval, $b \not\in J$. Thus $b \in I \setminus (I \cap J)$.
\end{proof}

\begin{lemma} \label{lem:i-leq-j}
   Let $I$ and $J$ be intervals with $J \leq I$. Then
   $J \setminus (I \cap J) \prec (I \cap J)$, and
   $(I \cap J) \prec I \setminus (I \cap J)$.
\end{lemma}

\begin{proof}
  First note that if either $A$ or $B$ is empty then $A \prec B$.
   Suppose $j \in J \setminus (I \cap J)$ and $i \in I \cap J$ with $i < j$.
   Since $J \leq I$, there is an $i' \in I$ with $j \leq i'$. Since $I$ is an interval, $j \in I$, which is a contradiction.
   Thus, for all $j \in J \setminus (I \cap J)$ and for all $i \in I \cap J$, $j \leq i$. That is, $J \setminus (I \cap J) \prec (I \cap J)$.
   Similarly, let $j \in I \cap J$ and $i \in I \setminus (I \cap J)$ with $i < j$.
   Again, since $J \leq I$, there is a $j' \in J$ with $j'\leq i$. Since $J$ is an interval, $i \in J$, which is a contradiction.
\end{proof}

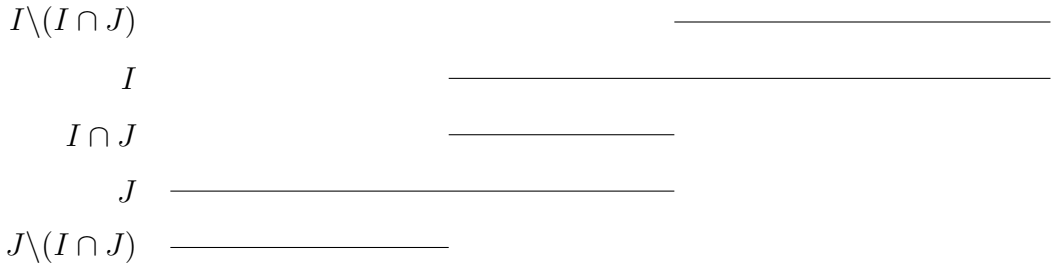
\begin{figure}[h!]
	\begin{tikzpicture}
	\draw (15,0) -- (20,0);
	\draw (12,-0.75) -- (20,-0.75);
	\draw (12,-1.5) -- (15,-1.5);
	\draw (8.3,-2.25) -- (15,-2.25);
	\draw (8.3, -3) -- (12,-3);

	\node[left] at (8,-0) {$I \backslash (I\cap J)$};
	\node[left] at (8,-0.75) {$I$};
	\node[left] at (8,-1.5) {$I\cap J$};
	\node[left] at (8,-2.25) {$J$};
	\node[left] at (8,-3) {$J \backslash (I \cap J)$};
	\end{tikzpicture}
\caption{The interval modules in Lemma~\ref{lem:all-intervals}, Lemma~\ref{lem:i-leq-j}, Proposition~\ref{prop:nonzero-map}, Lemma~\ref{lem:nonzero-map2}, and Corollary~\ref{cor:nonzero-map}.}
\end{figure}

\subsection{Nonzero maps of interval modules}
\label{sec:nonz-maps-interv}

In this section we characterize nonzero maps of interval modules.

\begin{proposition} \label{prop:nonzero-map}
   Let $I$ and $J$ be nonempty intervals. There is a nonzero map of persistence modules $f:I \to J$ if and only if $J \leq I$ and $I \cap J \neq \emptyset$.
\end{proposition}

\begin{proof}
   $(\To)$ Assume $f \neq 0$. Then there is an $a \in \mathbb{R}$ such that $0 \neq f_a:I(a) \to J(a)$. 
Without loss of generality, assume that $f_a = \Id$.
Thus $a \in I$ and $a \in J$. We need to check the conditions in Definition~\ref{def:interval-po}.

(\ref{it:int-po1}) For all $i \in I$ with $a \leq i$, the condition is satisfied by $a \in J$.
For all $i \in I$ with $i \leq a$, we have the following commutative diagram, \begin{equation*}
  \begin{tikzcd}[row sep=scriptsize]
    I(i) \ar[r,"\Id"] \ar[d,"f_i"'] & I(a) \ar[d,"f_a=\Id"] \\
    J(i) \ar[r] & J(a)
  \end{tikzcd}
\end{equation*}
which implies that $i \in J$, and thus $J(i\leq a) = \Id$, and therefore $f_i= \Id$.
(\ref{it:int-po2}) For all $j \in J$ with $j \leq a$, the condition is satisfied by $a \in I$.
For all $j \in J$ with $a \leq j$, we have the following commutative diagram, 
\begin{equation*}
  \begin{tikzcd}[row sep=scriptsize]
    I(a) \ar[r] \ar[d,"f_a=\Id"'] & I(j) \ar[d,"f_j"] \\
    J(a) \ar[r,"\Id"] & J(j)
  \end{tikzcd}
\end{equation*}
which implies that $j \in I$, $I(a \leq j) = \Id$, and $f_j= \Id$.

$(\From)$ Define $f:I \to J$ by $f_a=\Id$ if $a \in I \cap J$, and $f_a=0$ otherwise. We claim that $f$ is a natural transformation.
For $a \leq b$, we need to check that the diagram
\begin{equation*}
  \begin{tikzcd}[row sep=scriptsize]
    I(a) \ar[r] \ar[d,"f_a"] & I(b) \ar[d,"f_b"] \\
    J(a) \ar[r] & J(b)
  \end{tikzcd}
\end{equation*}
commutes. There are four cases to check.
If $a,b \in I \cap J$, then all four maps are the identity and thus the diagram commutes.
If $a,b \not\in I \cap J$ then both vertical maps are zero and thus the diagram commutes.

If $a \in I \cap J$ and $b \not\in I \cap J$ then by definition the left map is the identity and the right map is zero. 
If $b \in J$ then $b \not\in I$, which implies, since $I$ is an interval, that for all $c \geq b$, $c \not\in I$. But this contradicts Definition \ref{def:interval-po} (\ref{it:int-po1}). Therefore $b \not\in J$. Thus $J(b) = 0$ and hence the diagram commutes.

If $a \not\in I \cap J$ and $b \in I \cap J$, then $f_a=0$ and without loss of generality $f_b=\Id$. Again $a \in I$ implies $a \not\in J$, which implies that for all $c \leq a$, $c \not\in J$, which is a contradiction. Therefore $a \not\in I$ which implies that $I(a) = 0$ and thus the diagram commutes.
\end{proof}

\begin{lemma} \label{lem:nonzero-map2}
  Assume there is a nonzero map $f: I \to J$ of interval modules. Then (up to isomorphism) $f_a = \Id$ if $a \in I \cap J$ and $f_a = 0$ otherwise.
\end{lemma}

\begin{proof}
  Assume $f \neq 0$. The there is a $b \in I \cap J$ such that $f_b$ is nonzero. Without loss of generality, we may assume that $f_b = \Id$. Let $a \leq b \leq c \in I \cap J$. We have the following commutative diagram,
\begin{equation*}
  \begin{tikzcd}[row sep=scriptsize]
    I(a) \ar[r,"\Id"] \ar[d,"f_a"] & I(b) \ar[d,"f_b=\Id"] \ar[r,"\Id"] & I(c) \ar[d,"f_c"] \\
    J(a) \ar[r,"\Id"] & J(b) \ar[r,"\Id"] & J(c)
  \end{tikzcd}
\end{equation*}
which implies that $f_a = \Id$ and $f_c=\Id$.
Thus $f_a=\Id$ for all $a \in I \cap J$.

If $a \not\in I \cap J$ then either $I(a)=0$ or $J(a)=0$, which implies that $f_a=0$.
\end{proof}

\begin{corollary} \label{cor:nonzero-map}
  Let $f:I \to J$ be a nonzero map of interval modules. Then the image of $f$ is $I \cap J$, the kernel of $f$ is $I \setminus (I \cap J)$, the cokernel of $f$ is $J \setminus (I \cap J)$, and $f$ factors as follows.
\begin{equation*}
  \begin{tikzcd}[row sep=scriptsize]
    I \ar[rr,"f"] \ar[dr,twoheadrightarrow] & & J \\
    & I \cap J \ar[ur,hook]
  \end{tikzcd}
\end{equation*}
\end{corollary}

\subsection{Interleavings of interval modules}
\label{sec:interl-interv-modul}

In this section we characterize interleavings of interval modules.

\begin{definition}
  Let $I$ be an interval and $\eps \in \mathbb{R}$. Define the \emph{shifted interval} $I[\eps]$ by $x \in I[\eps]$ if and only if $x+\eps \in I$.
For example, $[a,b)[\eps] = [a-\eps,b-\eps)$. 
\end{definition}

The next lemma follows immediately from the definitions.

\begin{lemma}
  If $I$ is a nonempty interval and $\eps \geq 0$, then $I[\eps] \leq I$.
\end{lemma}

\begin{definition}
  Let $M$ be a persistence module and let $\eps \in \mathbb{R}$.
  We define the \emph{shifted persistence module} $M[\eps]$ by $M[\eps](a) = M(a + \eps)$ and $M[\eps](a \leq b) = M(a + \eps \leq b + \eps)$. That is, $M[\eps] = MT_{\eps}$.
\end{definition}

We remark that these two definitions are compatible. If $I$ is an interval module and $\eps \in \mathbb{R}$, then the shifted persistence module $I[\eps]$ is the interval module on the interval $I[\eps]$.
Also note that $0[\eps] = 0$.

Let $I$ be an interval and $\eps \geq 0$. If $I \cap I[\eps] \neq \emptyset$, we denote the corresponding nonzero map from Proposition~\ref{prop:nonzero-map} 
by $I^{(\eps)}:I \to I[\eps]$. 
If $I$ and $I[\eps]$ are disjoint, we denote the zero map 
by $I^{(\eps)}:I \to I[\eps]$. 
In either case, $I^{(\eps)} = I\eta_{\eps}$.

\begin{definition}
  Given a map of persistence modules $\alpha: M \to N$ and $\eps \in \mathbb{R}$, define $\alpha[\eps]: M[\eps] \to N[\eps]$ by $\alpha[\eps]_a = \alpha_{a+\eps}$. That is, $\alpha[\eps] = \alpha T_{\eps}$.
\end{definition}

As a special case of Definition~\ref{def:interleaving}, we have the following.

\begin{definition} \label{def:interleaving-interval}
  Let $I$ and $J$ be interval modules and $\eps \geq 0$. Then $I$ and $J$ are $\eps$-interleaved if there exist maps $\varphi: I \to J[\eps]$ and $\psi: J \to I[\eps]$ such that $\psi[\eps]\varphi = I^{(2\eps)}$ and $\varphi[\eps]\psi = J^{(2\eps)}$.
\end{definition}

\begin{lemma} \label{lem:interval-subset}
  If intervals satisfy $K \leq J \leq I$ then $I \cap K \subset J$.
\end{lemma}

\begin{proof}
  Let $x \in I \cap K$. 
  Then there is a $j \in J$ such that $x \leq j$. Also, there is a $j' \in J$ with $j' \leq x$. Since $J$ is an interval, $x \in J$.
\end{proof}

\begin{lemma} \label{lem:intersection}
  If $K \leq J \leq I$ then $I \cap K = (I \cap J) \cap (J \cap K)$.
\end{lemma}

\begin{proof}
  One direction is easy: $(I \cap J) \cap (J \cap K) = I \cap J \cap K \subset I \cap K$.
  The other direction follows from Lemma~\ref{lem:interval-subset}.
\end{proof}

\begin{proposition} \label{prop:interleaving}
  Let $I$ and $J$ be interval modules and $\eps \geq 0$. 
  If $J[\eps] \leq I$ and $I[\eps] \leq J$ then $I$ and $J$ are $\eps$-interleaved.
\end{proposition}

\begin{proof}



%
%
Define $\varphi:I \to J[\eps]$ by $\varphi_x = \Id$ if $x \in I \cap J[\eps]$ and $\varphi_x = 0$ otherwise. 
Similarly define $\psi: J \to I[\eps]$ by $\psi_x = \Id$ if $x \in J \cap I[\eps]$ and $\psi_x = 0$ otherwise.
We claim that these provide the desired $\eps$-interleaving.

First, $I^{(2\eps)}: I \to I[2\eps]$ is given by $I^{(2\eps)}_x = \Id$ if $x \in I \cap I[2\eps]$ and $I^{(2\eps)}_x = 0$ otherwise.
Next, $\psi[\eps]\varphi_x = \Id$ if $x \in (I \cap J[\eps]) \cap (J[\eps]\cap I[2\eps])$ and $\psi[\eps]\varphi_x = 0$ otherwise.
By Lemma~\ref{lem:intersection}, these maps are equal.
Similarly, $\varphi[\eps]\psi = J^{(2\eps)}$.
\end{proof}

Next, we define the erosion of a persistence module. Compare with~\cite{Patel:2016a,Puuska:2017}.

\begin{definition}
  Let $I$ be an interval or an interval module and $\eps \geq 0$. We define the \emph{$\eps$-erosion} of $I$ to be $I^{-\eps} = I[\eps] \cap I[-\eps]$.
\end{definition}

Note that $I[\eps] \leq I^{-\eps} \leq I[-\eps]$. See Figure~\ref{fig:erosion}.

\begin{figure}[h!]
	\begin{tikzpicture}
	\draw (10.3,0) -- (16.3,0);
	\draw (12.3,-0.75) -- (18.3,-0.75);
	\draw (12.3,-1.5) -- (14.3,-1.5);
	\draw (8.3,-2.25) -- (14.3,-2.25);
	\node[left] at (8,-0) {$I$};
	\node[left] at (8,-0.75) {$I[-\eps]$};
	\node[left] at (8,-1.5) {$I^{-\eps}$};
	\node[left] at (8,-2.25) {$I[\eps]$};
	\end{tikzpicture}
	\caption{An interval module and its erosion.}
	\label{fig:erosion}.
\end{figure}
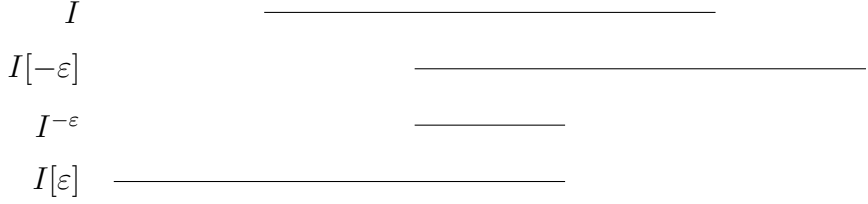

\begin{corollary} \label{cor:erosion}
  If $J[\eps] \leq I$ and $I[\eps] \leq J$ then $I^{-\eps} \subset J$ and $J^{-\eps} \subset I$.
\end{corollary}

\begin{proof}
  It follows from the assumptions that we also have $J \leq I[-\eps]$ and $I \leq J[-\eps]$. So we have $J[\eps] \leq I \leq J[-\eps]$ and $I[\eps] \leq J \leq I[-\eps]$.
  The result follows from Lemma~\ref{lem:interval-subset}.
\end{proof}

\begin{theorem} \label{thm:interleaving-interval}
  Let $I$ and $J$ be interval modules and $\eps \geq 0$. Then $I$ and $J$ are $\eps$-interleaved if only if
  $I^{-\eps} \subset J$ and $J^{-\eps} \subset I$.
\end{theorem}

\begin{proof}
  $(\Rightarrow)$ Let $\varphi$ and $\psi$ be an $\eps$-interleaving. If either $\varphi$ or $\psi$ are zero, then from Definition~\ref{def:interleaving-interval}, $I^{(2\eps)}$ and $J^{(2\eps)}$ are zero. It follows that $I^{-\eps}$ and $J^{-\eps}$ are both empty and the condition is satisfied. If both $\varphi$ and $\psi$ are nonzero, then by Proposition~\ref{prop:nonzero-map}, $J[\eps] \leq I$ and $I[\eps] \leq J$. The result follows from Corollary~\ref{cor:erosion}.

  $(\Leftarrow)$ We need to check four cases. (1) $I^{-\eps}$ and $J^{-\eps}$ are both empty. Then $I$ and $J$ are $\eps$-interleaved by $\varphi=0$ and $\psi=0$.

(2) $I^{-\eps}$ and $J^{-\eps}$ are both nonempty. Let $a \in I[\eps]$. Then there is an element $b \in I^{-\eps} \subset J$ with $a \leq b$. Let $b \in I$. Then there is an element $a \in I^{-\eps}[\eps] \subset J[\eps]$ with $a \leq b$. Let $a \in J[\eps]$. Then there is an element $b \in J^{-\eps} \subset I$ with $a \leq b$. Let $b \in J$. Then there is an element $a \in J^{-\eps}[\eps] \subset I[\eps]$ with $a \leq b$.
The result follows from Proposition~\ref{prop:interleaving}.

(3) $I^{-\eps}$ is nonempty and $J^{-\eps}$ is empty. 
Let $a \in I[\eps]$. Then there is $b \in I^{-\eps} \subset J$ with $a\leq b$.
Since $J$ is shorter than $I$, it follows that $I[\eps] \leq J$.
Let $b \in I$. Then there is $a \in I^{-\eps}[\eps] \subset J[\eps]$ with $a \leq b$.
Since $J$ is shorter than $I$, it follows that $J[\eps] \leq I$.
The result follows from Proposition~\ref{prop:interleaving}.

(4) is the same as the third case.
\end{proof}


\subsection{Neighborhoods of interval modules}
\label{sec:neighb-interv-modul}


Using Theorem~\ref{thm:interleaving-interval}, one obtains a complete characterization of the interval modules within distance $\eps$ of an interval module.

\begin{example}
	Consider the interval module $[a,b)$ and let $\eps \in [0,\frac{b-a}{2})$. 
Then an interval module $I$ is $\eps$-interleaved with $[a,b)$ if and only if $[a+\eps,b-\eps) \subset I \subset [a-\eps,b+\eps)$.
Furthermore $B_\eps([a,b))$ consists of those interval modules $I$ satisfying
	\[a-\eps < \inf I < a+\eps \quad \text{and} \quad b-\eps < \sup I < b+\eps. \] 
\end{example}

\begin{example}
  Consider the interval module $[a,b)$ and let $\eps \geq \frac{b-a}{2}$. Then an interval $I$ is $\eps$-interleaved with $[a,b)$ if and only if either $I \subset [a-\eps,b+\eps)$ or if for no $x \in \mathbb{R}$ do we have $[x-\eps,x+\eps] \subset I$.
Furthermore, $B_{\eps}([a,b))$ consists of those interval modules $I$ with either $a-\eps < \inf I$ and $\sup I < b+\eps$ or $\diam I < \eps$.
\end{example}

\begin{example}
  Consider the interval module $[a,\infty)$ and let $\eps \geq 0$. Then an interval module $I$ is $\eps$-interleaved with $[a,\infty)$ if and only if $[a+\eps,\infty) \subset I \subset [a-\eps,\infty)$. Furthermore $B_{\eps}([a,\infty))$ consists of interval modules $I$ satisfying
$a-\eps < \inf I < a + \eps$.
\end{example}


\end{document}